\documentclass[12pt]{article}

\usepackage{amsmath,amsthm,amsfonts,amssymb, color,graphicx}

\newtheorem{theorem}{Theorem}[section]

\newtheorem{example}[theorem]{Example}
\newtheorem{examples}[theorem]{Examples}

\newtheorem{lemma}[theorem]{Lemma}

\newtheorem{remark}[theorem]{Remark}

\allowdisplaybreaks[4]

\def\bD{\mathbb D}
\def\bE{\mathbb E}
\def\bN{\mathbb N}
\def\bP{\mathbb P}
\def\bR{\mathbb R}

\def\cB{\mathcal B}
\def\cD{\mathcal D}
\def\cF{\mathcal F}
\def\cH{\mathcal H}

\def\cA{\mathcal{A}}
\def\cD{\mathcal{D}}

\def\cF{\mathcal{F}}
\def\cG{\mathcal{G}}

\def\cH{\mathcal{H}}
\def\cP{\mathcal{P}}

\def\cS{\mathcal{S}}
\def\cT{\mathcal{T}}

\def\fX{\mathfrak{X}}

\def\bD{\mathbb{D}}
\def\bE{\mathbb{E}}
\def\bR{\mathbb{R}}

\def\e{\varepsilon}

\begin{document}

\title{Spatial integral of the solution to
hyperbolic Anderson model
with time-independent noise}

\author{Raluca M. Balan\footnote{Corresponding author. University of Ottawa, Department of Mathematics and Statistics, STEM Building, 150 Louis-Pasteur Private,
Ottawa, ON, K1N 6N5, Canada. E-mail: rbalan@uottawa.ca. Research
supported by a grant from the Natural Sciences and Engineering
Research Council of Canada.} \and   Wangjun Yuan\footnote{University of Ottawa, Department of Mathematics and Statistics, STEM Building, 150 Louis-Pasteur Private,
Ottawa, ON, K1N 6N5, Canada. E-mail: wyuan2@uottawa.ca.}}

\date{January 15, 2022}
\maketitle

\begin{abstract}
\noindent In this article, we study the asymptotic behavior of the spatial integral of the solution to the hyperbolic Anderson model in dimension $d\leq 2$, as the domain of the integral gets large (for fixed time $t$). This equation is driven by a spatially homogeneous Gaussian noise, whose covariance function is either integrable, or is given by the Riesz kernel. The novelty is that the noise does not depend on time, which means that It\^o's martingale theory for stochastic integration cannot be used. Using a combination of Malliavin calculus with Stein's method, we show that with proper normalization and centering, the spatial integral of the solution converges to a standard normal distribution, by estimating the speed of this convergence in the total variation distance. We also prove the corresponding functional limit theorem for the spatial integral process.
\end{abstract}



\section{Introduction}

Stochastic partial differential equations (SPDEs) are mathematical models used for physical phenomena which are subject to random perturbations. These perturbations are typically described by a collection of random variables which constitute the noise. In general, the existence and behaviour of the solution of an SPDE is strongly influenced by the temporal component of the noise, which usually resembles a well-understood classical process, such as Brownian motion, fractional Brownian motion, or a L\'evy process. But it is also possible that the noise has no temporal component at all, which will be the case for the equation considered in the present article.

In the recent years, there has been a lot of interest in studying the asymptotic behavior as $R \to \infty$ of the spatial integral $I_R(t)=\int_{B_R} u(t,x)dx$ associated with the random field solution $u$ of an SPDE (as defined in Walsh' lecture notes \cite{walsh86}), where $B_R$ is the ball of radius $R$ centered at $0$.
This line of investigations was initiated in article \cite{HNV20} for the stochastic heat equation in dimension $d=1$, driven by a spatially homogeneous Gaussian noise which behaves in time like Brownian motion (i.e. it is white in time).
In this equation, the initial condition is constant and the noise is multiplied by $\sigma(u)$, where $\sigma$ is a Lipschitz function. Combining tools from Malliavin calculus with Stein's method for normal approximations, the authors of \cite{HNV20} proved that for fixed $t>0$, $I_R(t)$ converges (with proper normalization and centering) to the $N(0,1)$ distribution as $R \to \infty$, by estimating the speed of this convergence in the total variation distance. They also proved the corresponding functional limit theorem for the integral process $\{I_R(t);t\geq 0\}$ in the space $C[0,\infty)$ of continuous functions on $[0,\infty)$.
Further developments and extensions for the heat equation with the same type of noise can be found in \cite{CKNP1,CKNP2,HNVZ}. The parabolic Anderson model (corresponding to the case $\sigma(u)=u)$ with the same noise and delta initial condition was studied in \cite{CKNP3}. The same problem for the fractional heat equation (in which the Laplacian is replaced by its fractional power) has been considered in \cite{ANTV}. The case of the parabolic Anderson model driven by a Gaussian noise colored in time was treated in \cite{NZ-EJP,NXZ}, and the same model with rough noise in space appeared in \cite{NSZ}.

There are several articles dedicated to this problem for the stochastic wave equation in dimension $d\leq 2$, with constant initial condition and a Lipschitz function $\sigma(u)$ multiplying the noise. The case of the white noise in time with spatial covariance given by the Riesz kernel was considered in \cite{DNZ20,BNZ21} for $d=1$ and $d=2$ respectively, while the case of an integrable spatial covariance function was studied in \cite{NZ20}. The hyperbolic Anderson model (for which $\sigma(u)=u$) with the colored noise in time was examined in \cite{BNQZ}.

\medskip

In this article, we consider the following hyperbolic Anderson model:
\begin{align}
\label{ham2}
	\begin{cases}
		\dfrac{\partial^2 u}{\partial t^2} (t,x)
		= \Delta u(t,x) + u(t,x) \dot{W}(x), \
		t>0, \ x \in \bR^d, (d\leq 2)\\
		u(0,x) = 1, \ \dfrac{\partial u}{\partial t} (0,x) = 0.
	\end{cases}
\end{align}
The novelty stems from the fact that the noise is time-independent, i.e. it is given by a zero-mean Gaussian process $\{W(\varphi);\varphi \in \cD(\bR^d)\}$
with covariance:
\[
\bE[W(\varphi)W(\psi)]
=\int_{(\bR^d)^2} \gamma(x-y) \varphi(x) \psi(y)dxdy
=:\langle \varphi,\psi\rangle_{\cP_0},
\]
where $\cD(\bR^d)$ is the space of infinitely differentiable functions on $\bR^d$, with compact support. We assume that the noise is defined on a complete probability space $(\Omega,\cF,\bP)$, and
we denote by $\|\cdot\|_p$ the $L^p(\Omega)$-norm for any $p>0$. We will work with a ``strong'' solution, i.e. a solution defined on this fixed probability space.

The parabolic Anderson model with Gaussian time-independent noise as above:
\begin{align}
\label{pam}
	\begin{cases}
		\dfrac{\partial u}{\partial t} (t,x)
		= \frac{1}{2}\Delta u(t,x) + u(t,x) \dot{W}(x), \
		t>0, \ x \in \bR^d, (d\geq 2)\\
		u(0,x) = 1, 
	\end{cases}
\end{align}
appeared for the first time in article \cite{hu01}, the solution being defined in the Skorohod sense. H\"older continuity of the sample paths and exponential bounds for the moments (known as intermittency properties) were obtained in \cite{HHNT} of both Skorohod and Stratonovich solutions using their respective Feynman-Kac representations. The exact asymptotic behaviour of the moments of the Skorohod solution for the same model was obtained in \cite{xchen17} (taking $\alpha_0=0$), under additional assumptions on $\gamma$. New properties of the Skorohod solution of \eqref{pam} in dimension $d=1$ have been recently discovered in \cite{kim-scorolli, scorolli}.

By contrast, the hyperbolic model \eqref{ham2} is far less studied. In fact, we are aware of only two references for this model, both of them quite recent and involving the first author of this paper. More precisely, the exact asymptotic behaviour of the moments of the Skorohod solution of \eqref{ham2} has been obtained in \cite{BCC} under the same assumptions on $\gamma$ as in \cite{xchen17}, while the Stratonovich solution of equation \eqref{ham2} has been examined in \cite{B21}. 

\medskip

We give now few more details about the noise $W$.
We assume that the function $\gamma:\bR^d \to [0,\infty]$ is non-negative-definite in the sense of distributions. By the Bochner-Schwartz theorem, there exists a tempered measure $\mu$ on $\bR^d$ such that  $\gamma=\cF \mu$ in $\cS'(\bR^d)$:
\[
\int_{\bR^d} \varphi(x) \gamma(x)dx=\int_{\bR^d}\cF \varphi(\xi)\mu(d\xi), \quad \mbox{for any $\varphi\in \cS(\bR^d)$}.
\]
We say that $\mu$ is the {\em spectral measure} of $\gamma$ (and of the noise $W$). Consequently,
\[
\int_{\bR^d} \int_{\bR^d} \varphi(x) \psi(y)\gamma(x-y)dxdy=\int_{\bR^d}\cF \varphi(\xi) \overline{\cF \psi(\xi)}\mu(d\xi), \quad \mbox{for any $\varphi,\psi \in \cS(\bR^d)$}.
\]
Here $\cF \varphi(\xi)=\int_{\bR^d}e^{-i \xi \cdot x}\varphi(x)dx$ is the Fourier transform of $\varphi$ and $\cS(\bR^d)$ is the set of rapidly decreasing functions. 
If $\mu(\bR^d)<\infty$, then $\gamma(x)=\int_{\bR^d}e^{-i\xi \cdot x}\mu(d\xi)$ for any $x \in \bR^d$.

The noise is stationary (or homogeneous) in space, i.e. the covariance of the random distribution
$\{W(\varphi);\varphi \in \cD(\bR^d)\}$ is invariant under translations:
\[
\bE[W(\tau_{h}\varphi)W(\tau_h \psi)]=\bE[W(\varphi)W(\psi)] \quad \mbox{for any} \quad h \in \bR^d,
\]
where $(\tau_h \varphi)(x)=\varphi(x+h)$ for all $x \in \bR^d$. The concept of stationary random distribution (not necessarily Gaussian) goes back to the 1950's, as it was introduced by It\^o in \cite{ito54} for $d=1$, and was extended to $d\geq 1$ by Yaglom in \cite{yaglom57}, who called it a ``homogeneous generalized random field''.
In the Gaussian case, this type of covariance structure became very popular for the noise perturbing an SPDE only after the publication of Dalang's seminal article \cite{dalang99}, in which the noise is white in time.

We assume that the spectral measure $\mu$ of the noise satisfies {\em Dalang's condition}:
\begin{equation}
\label{D-cond}
C_{\mu}:=\int_{\bR^d}\frac{1}{1+|\xi|^2}\mu(d\xi)<\infty.
\tag{D}
\end{equation}
Note that this condition always holds for $d=1$ (see Remark 10 of \cite{dalang99}).

\medskip

Below are some examples of pairs $(\gamma,\mu)$. In these examples, $\mu$ has density function $g$.

\begin{examples}
\label{ex}
{\rm
\begin{enumerate}
\item {\em (Heat kernel)} $\gamma(x)=(2\pi a)^{-d/2}e^{-|x|^2/(2a)}$, $g(\xi)=e^{-a|\xi|^2}$ $(a>0)$
\item {\em (Poisson kernel)}
$\gamma(x)=c_d a (a^2+|x|^2)^{-\frac{d+1}{2}}$, $g(\xi)=e^{-a|\xi|}$
$(a>0)$

\item {\em (Riesz kernel)}
$\gamma(x)=|x|^{-\beta}$, $g(\xi)=C_{d,\beta}|\xi|^{-(d-\beta)}$
$(\beta \in (0,d))$

\item {\em (Bessel kernel)}
$\gamma(x)=\frac{1}{\Gamma(\alpha)}\int_0^{\infty}
t^{\alpha-1}(4\pi t)^{-d/2}e^{-t-|x|^2/(4t)} dt$, $g(\xi)=(1+|\xi|^2)^{-\alpha/2}$ \\
($\alpha>0$)
\item {\em (Fractional kernel)}
$\gamma(x)=\prod_{i=1}^{d}\alpha_{H_i}|x_i|^{2H_i-2}$, $g(\xi)=\prod_{i=1}^{d}c_{H_i} |\xi_i|^{1-2H_i}$ with
$x=(x_1,\ldots,x_d)$, $\xi=(\xi_1,\ldots,\xi_d)$
$\alpha_{H}=H(2H-1)$, $c_{H}=\frac{\Gamma(2H+1)\sin(\pi H)}{2\pi}$
$(H_i  \in (\frac{1}{2},1))$. Then $\{W(x)=W(1_{[0,x]})\}_{x \in \bR^d}$ is a fractional Brownian sheet with indices $H_1,\ldots,H_d$.
\end{enumerate}
}
\end{examples}

Let $\cP_0$ be the completion of $\cD(\bR^d)$ with respect to 
$\langle \cdot, \cdot \rangle_{\cP_0}$. By the isometry property, the map $\cD(\bR^d) \ni \varphi \mapsto W(\varphi)\in L^2(\Omega)$ can be extended $\cP_0$. Then $W=\{W(\varphi);\varphi \in \cP_0\}$ is an isonormal Gaussian process as in Malliavin calculus (see \cite{nualart06}).

The Hilbert space $\cP_0$ may contain tempered distributions. By Theorem 3.5 of \cite{BGP12}, if $\mu$ has density function $g$, then
\[
\cP_0 \subset \mathcal{U}_0:=\{S \in \cS'(\bR^d); \cF S \ \mbox{is a function}, \int_{\bR^d}|\cF S(\xi)|^2 \mu(d\xi)<\infty \},
\]
and $\cP_0=\mathcal{U}_0$ if $1/g$ is tempered, i.e. $\int_{g>0} (1+|\xi|^2)^{-k} [g(\xi)]^{-1}d\xi<\infty$ for some $k \in \bN$. In Example \ref{ex}.3, $\cP_0 \subset H^{-(d-\beta)/2}(\bR^d)$, where $H^{r}(\bR^d)$ is the fractional Sobolev space of order $r$, and (D) holds if and only if $\beta<2$. In Example \ref{ex}.4, $\cP_0=H^{-\alpha/2}(\bR^d)$, and (D) holds if and only if $d-\alpha<2$.

\begin{example}[{\em white noise}]
{\rm
We consider also the case when $W$ is white noise, i.e. $$\bE[W(\varphi)W(\psi)]=\int_{\bR^d}\varphi(x)\psi(x)dx.$$ In this case, $\cP_0=L^2(\bR^d)$, $\gamma=\delta_0$ (formally), $\mu(d\xi)=(2\pi)^{-d}d\xi$ and
$\{W(x)=W(1_{[0,x]})\}_{x \in \bR^d}$ is a Brownian sheet. Formally, the white noise case corresponds to Example \ref{ex}.3 with $\beta=d$. Obviously, \eqref{D-cond} holds if and only if $d=1$; in this case, $C_{\mu}=1/2$.
}
\end{example}

We now introduce the concept of solution.
A process $u=u(t,x);t\geq 0,x \in \bR^d\}$ is a (Skorohod) {\bf solution} to equation \eqref{ham2} if it satisfies the following integral equation:
\[
u(t,x)=1+\int_0^t \int_{\bR^d}G_{t-s}(x-y)u(s,y)W(\delta y)ds,
\]
where $W(\delta y)$ denotes the Skorokod integral with respect to $W$, and $G$ is the fundamental solution to the deterministic wave equation with the same initial conditions as \eqref{ham2}:
\begin{align}
\label{Green}
G_t(x) :=
\begin{cases}
	\dfrac{1}{2} \mathbf{1}_{\{ | x| < t\}} \quad &\text{if $d=1$}; \\
	\dfrac{1}{2\pi  \sqrt{ t^2 - |x|^2}}  \mathbf{1}_{\{ |x| < t\}} \quad &\text{if $d=2$},
\end{cases}
\end{align}
for any $t>0$ and $x \in \bR^d$, with $|\cdot|$ being the Euclidean norm.

The goal of the present paper is to investigate the asymptotic behaviour as $R \to \infty$ of the centered spatial integral:
\[
F_{R}(t)=\int_{B_R} \big(u(t,x)-1\big)dx,
\]
where $B_R=\{x \in \bR^d;|x|<R\}$. Letting $\sigma_{R}^2(t)={\rm Var}\big(F_R(t)\big)$, we will show that
\[
\frac{F_R(t)}{\sigma_R(t)} \stackrel{d}{\to} Z \sim N(0,1) \quad \mbox{as $R \to \infty$},
\]
by estimating the speed of this convergence in the total variation distance $d_{TV}$. Recall that $d_{\rm TV}(X,Y)=\sup_{B \in \cB(\bR)}|\mu_X(B)-\mu_Y(B)|$ for random variables $X,Y$ with respective laws $\mu_X,\mu_Y$, and $d_{\rm TV}(X_n,X) \to 0$ as $n\to \infty$ implies that $X_n \stackrel{d}{\to} X$ as $n\to \infty$.

In order to do this, we follow the same general strategy as in \cite{BNQZ}, namely we first identify the order of magnitude of $\sigma_R^2(t)$, and then use the bound given by Proposition 1.8 of \cite{BNQZ} for the distance $d_{TV}(F_R(t)/\sigma_R(t),Z)$, which is valid also for the time-independent noise.\footnote{
A different (but longer) argument for estimating the distance $d_{TV}(F_R(t)/\sigma_R(t),Z)$ can be found in the first version of this article (available on arXiv:2201.02319). This argument is based on the classical Stein-Malliavin bound (as in the original article  \cite{HNV20}) and illustrates the challenges of working with the time-independent noise compared with the white noise in time. For instance, there is no Clarke-Ocone formula for the noise $W$, and key results from It\^o's martingale theory (such as Burkholder-Davis-Gundy inequality) cannot be used simply because there is no martingale.}

A key idea, which is common to all references who studied this problem, is to show that the moments of the first and second Malliavin derivatives of $u(t,x)$ are dominated, respectively, by the first two chaos kernels $f_1(\cdot,x;t)$ and $f_2(\cdot,x;t)$ which appear in the chaos expansion of the solution. We will achieve this too, in relations \eqref{D-bound} and \eqref{bound-D2} below.

\medskip

When $d=2$, we will impose the following hypothesis:
\begin{align*}
{\bf (H1)} \begin{cases}
  &\text{(\texttt{a})   $\gamma\in L^\ell(\bR^2)$ for some $\ell\in(1,\infty)$; \mbox{or}}
  \\
& \text{(\texttt{b})   $\gamma(x)= |x|^{-\beta}$ for some $\beta\in(0,2)$,}
\end{cases}
\end{align*}

Under this assumption, if we define the constant $q\in(1/2, 1)$ by
\begin{align}
\label{def-q}
q=
\begin{cases}
\ell/(2\ell-1) & \mbox{in case  \rm (\texttt{a})}, \\
2/(4-\beta)& \mbox{in case \rm (\texttt{b})}.
\end{cases}
\end{align}
then $L^{2q}(\bR^{2n}) \subset \cP_{0}^{\otimes n}$ and for any $f,g \in L^{2q}(\bR^{2n})$,
\begin{equation}
\label{embed}
\langle f,g \rangle_{\cP_0^{\otimes n}} \leq C^n \|f\|_{L^{2q}(\bR^{nd})} \|g\|_{L^{2q}(\bR^{nd})}
\end{equation}
where the constant $C>0$ depends only on $\gamma$ (see Lemma 2.3.(1) of \cite{BNQZ}). This inequality will play an important role in the present paper.
In the case $d=1$, we do not need a hypothesis similar to {\bf (H1)}, since the function $G$ has a very simple form.

\medskip

It can be proved that for any $t>0$ and $s>0$ fixed, the covariance
\[
\bE[(u(t,x)-1)(u(s,y)-1)]:=\rho_{t,s}(x-y)
\]
is non-negative and depends only on $x-y$ (see Remark \ref{rem-cov} below). In particular, $\{u(t,x)\}_{x\in \bR^d}$ is a positively-correlated
stationary process with covariance $\rho_t=\rho_{t,t}$.

\medskip

We give few comments about the notation. We write $f(R) \sim g(R)$ if $f(R)/g(R) \to 1$ as $R \to \infty$. We let $\omega_d$ be the Lebesque measure of $B_1$, i.e. $\omega_1=2$ and $\omega_2=\pi$ if $d=2$. We denote by $C[0,\infty)$ the space of continuous functions $f:[0,\infty) \to \bR$, equipped with the uniform convergence on compact sets.

\medskip

We are now ready to state the main results of this article, which correspond to the two cases of Hypothesis {\bf (H1)} when $d=2$.

The first result covers the case when $\gamma$ is an integrable function, or the noise is white (in space) and $d=1$. The analogue result for the white noise in time is given in \cite{NZ20}.

\begin{theorem}
\label{CLT-integr}
Suppose that $\gamma$ is non-negative and non-negative definite, $\gamma \in L^{1}(\bR^d)$ and the spectral measure $\mu$ satisfies (D), or the noise is white and $d=1$. If $d=2$, in parts (ii)-(iii) below, we assume in addition that $\gamma \in L^{\ell}(\bR^2)$ for some $\ell>1$. Then:

(i) for any $t>0$ and $s>0$,
\[
\bE[F_R(t)F_R(s)] \sim K(t,s) R^{d} \ as \ R \to \infty, \ where \
K(t,s):=\omega_d\int_{\bR^d}\rho_{t,s}(z)dz<\infty,
\]
and in particular $\sigma_{R}^2(t) \sim K(t,t) R^{d}$ as $R \to \infty$;

(ii) for any $t>0$,
\[
d_{\rm TV}\left( \frac{F_R(t)}{\sigma_R(t)},Z\right) \leq C_t R^{-d/2},
\]
where $C_{t}>0$ is a constant depending on $t$;

(iii) there exists a continuous modification of the process $\{R^{-d/2} F_R(t)\}_{t\geq 0}$ which converges in distribution in $C[0,\infty)$ as $R \to \infty$, to a zero-mean Gaussian process $\{\cG(t)\}_{t\geq 0}$ with covariance
$ \bE[\cG(t)\cG(s)]=K(t,s)$.
\end{theorem}

The second result covers the case when $\gamma$ is the Riesz kernel.
The counterpart of this result for the white noise in time can be found in \cite{DNZ20,BNZ21} for $d=1$, respectively $d=2$.

\begin{theorem}
	\label{CLT-Riesz}
	Suppose that $\gamma(x)=|x|^{-\beta}$ where $\beta \in (0,d \wedge 2)$. Then:

	(i) for any $t>0$ and $s>0$
	\[
\bE[F_R(t)F_R(s)] \sim K'(t,s) R^{2d-\beta}  \ as \ R \to \infty, \ where \
	K'(t,s) := \dfrac{t^2 s^2}{4} \int_{B_1^2} |x-x'|^{-\beta} dxdx',
\]
and in particular $\sigma_{R}^2(t) \sim K'(t,t) R^{2d-\beta}$;

	(ii) for any $t>0$,
	\[
	d_{\rm TV}\left( \frac{F_R(t)}{\sigma_R(t)},Z\right) \leq C_t' R^{-\beta/2}
\]
where $C_t'>0$ is a constant depending on $t$;

	(iii) there exists a continuous modification of the process $\{R^{-d+\beta/2} F_R(t)\}_{t\geq 0}$ which converges in distribution in $C[0,\infty)$  as $R \to \infty$, to a zero-mean Gaussian process $\{\cG(t)\}_{t\geq 0}$ with covariance
$\bE[\cG(t)\cG(s)]= K'(t,s)$.
\end{theorem}

The article is organized as follows. In Section \ref{section-exist}, we review some basic facts about Malliavin calculus, we prove the existence of solution under (D), and we compute the covariance of the solution. In Section \ref{sec:Malliavin derivative}, we give the key moment estimates for the first and second Malliavin derivatives of the solution. The proofs of Theorems \ref{CLT-integr} and \ref{CLT-Riesz}
are given in Sections \ref{section-CLT1} and \ref{section-CLT2}, respectively. The appendix contains some auxiliary results.

To simplify the writing, throughout the article we will use the convention:
 \begin{align}
 \text{$G_t(x) =0$ when $t\leq 0$.}
 \label{rule1}
 \end{align}

\section{Skorohod solution}
\label{section-exist}

In this section, we review some elements of Malliavin calculus and give some basic properties of the solution. In particular, we show that under condition (D), equation \eqref{ham2} has a unique (Skorohod) solution, a result which was stated in \cite{BCC} without proof (see Theorem 2.2 of \cite{BCC}). We refer the reader to \cite{nualart06,NN} for more details about Malliavin calculus.

\medskip

Since $W=\{W(\varphi);\varphi \in \cP_0\}$ is an isonormal Gaussian process,
every square-integrable random variable $F$
which is measurable with respect to $W$ has the Wiener chaos expansion:
\begin{equation}
\label{F-chaos}
F=E(F)+\sum_{n \geq 1}I_n(f_n) \quad \mbox{for some} \quad f_n \in \cP_0^{\otimes n},
\end{equation}
where $\cP_0^{\otimes n}$ is the
$n$-th tensor product of $\cP_0$ and $I_n$
is the multiple Wiener integral with respect to $W$.
By the orthogonality of the Wiener chaos spaces,
\[
 E[I_n(f)I_m(g)]=\left\{
\begin{array}{ll} n! \, \langle \widetilde{f}, \widetilde{g} \rangle_{\cP_{0}^{\otimes n}} & \mbox{if $n=m$} \\
0 & \mbox{if $n \not=m$}
\end{array} \right.
\]
where $\widetilde{f}$ is the symmetrization of $f$ in all $n$
variables:
$$\widetilde{f}(x_1,\ldots,x_n)=\frac{1}{n!}\sum_{\rho \in S_n}f(x_{\rho(1)},\ldots,x_{\rho(n)}),$$
and $S_n$ is the set of all permutations of $\{1,
\ldots,n\}$. It can be proved that:
\begin{equation}
\label{rough-bound}
\|\widetilde f\|_{\cP_0^{\otimes n}} \leq \|f\|_{\cP_0^{\otimes n}},
\end{equation}
an inequality will be used several times below.
If $F$ has the chaos expansion \eqref{F-chaos}, then
$$E|F|^2=\sum_{n \geq 0}E|I_n(f_n)|^2=\sum_{n \geq 0}n! \, \|\widetilde{f}_n\|_{\cH^{\otimes n}}^{2}.$$

Let $\cS$ be the class of ``smooth'' random variables, i.e variables of the form
\begin{equation}
\label{form-F}F=f(W(\varphi_1),\ldots, W(\varphi_n)),
\end{equation} where $f \in C_{b}^{\infty}(\bR^n)$, $\varphi_i \in \cP_0$, $n \geq 1$, and
$C_b^{\infty}(\bR^n)$ is the class of bounded $C^{\infty}$-functions
on $\bR^n$, whose partial derivatives of all orders are bounded. The
{\em Malliavin derivative} of $F$ of the form (\ref{form-F}) is the
$\cP_0$-valued random variable given by:
$$DF:=\sum_{i=1}^{n}\frac{\partial f}{\partial x_i}(W(\varphi_1),\ldots,
W(\varphi_n))\varphi_i.$$ We endow $\cS$ with the norm
$\|F\|_{\bD^{1,2}}:=(E|F|^2)^{1/2}+(E\|D F \|_{\cP_0}^{2})^{1/2}$. The
operator $D$ can be extended to the space $\bD^{1,2}$, the
completion of $\cS$ with respect to $\|\cdot \|_{\bD^{1,2}}$.

The {\em divergence operator} $\delta$ is the adjoint of
the operator $D$. The domain of $\delta$, denoted by $\mbox{Dom} \
\delta$, is the set of $u \in L^2(\Omega;\cP_0)$ such that
$$|E \langle DF,u \rangle_{\cH}| \leq c (E|F|^2)^{1/2}, \quad \forall F \in \bD^{1,2},$$
where $c$ is a constant depending on $u$. If $u \in {\rm Dom} \
\delta$, then $\delta(u)$ is the element of $L^2(\Omega)$
characterized by the following duality relation:
\begin{equation}
\label{duality}
E(F \delta(u))=E\langle DF,u \rangle_{\cP_0}, \quad
\forall F \in \bD^{1,2}.
 \end{equation}
In particular, $E(\delta(u))=0$. If $u \in \mbox{Dom} \ \delta$, we
use the notation
$$\delta(u)=\int_{\bR^d}u(x) W(\delta x),$$
and we say that $\delta(u)$ is the {\em Skorohod integral} of $u$
with respect to $W$.

If $F$ has the chaos expansion \eqref{F-chaos}, we define the {\em Ornstein-Uhlenbeck generator}
\[
LF=\sum_{n\geq 1}n I_n(f_n)
\]
provided that the series converges in $L^2(\Omega)$.
It can be proved that
$F \in {\rm Dom}\ L$ if and only if $F \in \bD^{1,2}$ and $DF \in {\rm Dom} \ \delta$; in this case, $LF=-\delta (D F)$.
The pseudo-inverse $L^{-1}$ of $L$ is defined by
\[
L^{-1}F=\sum_{n\geq 1}\frac{1}{n} I_n(f_n).
\]
For any $F \in \bD^{1,2}$ with $\bE(F)=0$, the process $u=-D L^{-1}F$ belongs to ${\rm Dom} \ \delta$
and
\begin{equation}
\label{F-deltaD}
F=\delta(-D L^{-1} F).
\end{equation}
(see e.g. Proposition 6.5.1 of \cite{NN}).

\medskip

We return now to our problem. The solution to equation \eqref{ham2} exists if and only if the series $\sum_{n\geq 1}I_n(f_n(\cdot,x;t)) $ converges in $L^2(\Omega)$, where the kernel $f_n(\cdot,x;t)$ is given by:
\[
f_n(x_1,\ldots,x_n,x;t)=\int_{T_n(t)} G_{t-t_n}(x-x_n)\ldots G_{t_2-t_1}(x_2-x_1)dt_1 \ldots dt_n,
\]
with $T_n(t)=\{(t_1,\ldots,t_n);0<t_1<\ldots<t_n<t\}$.
In this case,
\begin{align}
\label{eq-chaos expansion}
	u(t,x)=1+\sum_{n\geq 1}I_n(f_n(\cdot,x;t)),
\end{align}
and
\[
\bE|u(t,x)|^2=\sum_{n\geq 1}n!\|\widetilde{f}_n(\cdot,x;t)\|_{\cP_0^{\otimes n}}^2.
\]

We have following  result.
\begin{theorem}
\label{exist-th}
Assume that $\gamma$ is non-negative and non-negative definite and the spectral measure $\mu$ satisfies (D), or the noise is white and $d=1$. Then for any $t>0$ and $x \in \bR^d$,
\begin{equation}
\label{e-fn}
\|f_n(\cdot,x;t)\|_{\cP_0^{\otimes n}}^2 \leq \left(\frac{t^n}{n!} \right)^2 D_t^n C_{\mu}^n,
\end{equation}
where $D_t=2(t^2 \vee 1)$ and $C_{\mu}$ is given by \eqref{D-cond}, or $C_{\mu}=1/2$ if the noise is white. Consequently,
equation \eqref{ham2} has a unique (Skorohod) solution, and for any $p\geq 2$ and $T>0$,
\begin{equation}
\label{e-u}
\sup_{(t,x) \in [0,T] \times \bR^d}\|u(t,x)\|_p<\infty.
\end{equation}
\end{theorem}

\begin{proof}
 By the Cauchy-Schwarz inequality,
\begin{align*}
 |\cF f_n(\cdot,x;t)(\xi_1,\ldots,\xi_n) |^2&=\left|\int_{T_n(t)} \prod_{j=1}^{n}\cF G_{t_{j+1}-t_j}(\xi_1+\ldots+\xi_j) dt_1 \ldots dt_n\right|^2 \\
 & \leq \frac{t^n}{n!}\int_{T_n(t)} \prod_{j=1}^{n}|\cF G_{t_{j+1}-t_j}(\xi_1+\ldots+\xi_j)|^2 dt_1 \ldots dt_n,
\end{align*}
where $t_{n+1}=t$. Hence,
\begin{equation}
\label{bound-fn}
\|f_n(\cdot,x;t)\|_{\cP_0^{\otimes n}}^2 =
\int_{(\bR^d)^n}|\cF f_n(\cdot,x;t)(\xi_1,\ldots,\xi_n) |^2 \mu(d \xi_1) \ldots \mu(d\xi_n)\leq \frac{t^n}{n!} J_n(t),
\end{equation}
where
\begin{equation}
\label{def-Jn}
J_n(t)=\int_{T_n(t)} \int_{(\bR^d)^n} \prod_{j=1}^{n}|\cF G_{t_{j+1}-t_j}(\xi_1+\ldots+\xi_j)|^2 \mu(d \xi_1) \ldots \mu(d\xi_n)dt_1 \ldots dt_n.
\end{equation}

Using the fact that $|\cF G_t(\xi)|^2 =\frac{\sin^2(t|\xi|)}{|\xi|^2}\leq D_t \frac{1}{1+|\xi|^2}$ with $D_t=2(t^2 \vee 1)$, and
\begin{equation}
\label{max-principle}
\sup_{\eta \in \bR^d}\int_{\bR^d}\frac{1}{1+|\xi+\eta|^2}\mu(d\xi)=\int_{\bR^d}
\frac{1}{1+|\xi|^2}\mu(d\xi)=C_{\mu},
\end{equation}
we obtain that
\begin{equation}
\label{Jn-bound}
J_n(t) \leq D_t^n C_{\mu}^n \frac{t^n}{n!}.
\end{equation}

Relation \eqref{e-fn} follows.
Using the rough bound \eqref{rough-bound}, we get:
\[
\sum_{n\geq 1}n! \|\widetilde{f}_n(\cdot,x;t)\|_{\cP_0^{\otimes n}}^2
\leq \sum_{n\geq 1}n! \|f_n(\cdot,x;t)\|_{\cP_0^{\otimes n}}^2 \leq \sum_{n\geq 1}
\frac{t^{2n}}{n!} D_t^n C_{\mu}^n<\infty.
\]
This proves the existence of solution.
To estimate its moments, we use hypercontractivity:
\[
\|u(t,x)\|_p \leq \sum_{n\geq 0}(p-1)^{n/2} (n!)^{1/2}\|\widetilde{f}_n(\cdot,x;t)\|_{\cP_0^{\otimes n}}
\leq \sum_{n\geq 0}(p-1)^{n/2}  \frac{t^{n}}{(n!)^{1/2}}(D_t C_{\mu})^{n/2}.
\]
Relation \eqref{e-u} follows, since the constant $D_t$ is increasing in $t$.
\end{proof}

\begin{remark}
{\rm Theorem \ref{exist-th} remains valid in any dimension $d\geq 1$.
}
\end{remark}

\begin{remark}[white noise]
{\rm By Theorem 3.1 of \cite{BCC}, we know that equation \eqref{ham2} driven by white noise has a unique (Skorohod) solution which satisfies \eqref{e-u} also in dimension $d=2$, although
condition (D) does not hold in this case. Unfortunately, in the case of the white noise in dimension $d=2$,
we could not prove the key estimate \eqref{D-bound} below for the Malliavin derivative $D_zu(t,x)$, the main difficulty being that $G_t^2$ is not integrable.
}
\end{remark}

\begin{remark}[Comparison with white noise in time]
{\rm Consider the hyperbolic Anderson model with Gaussian noise $\fX$ which is white noise in time and has the same spatial covariance structure as $W$:
\begin{align}
\label{ham1}
	\begin{cases}
		\dfrac{\partial^2 v}{\partial t^2} (t,x)
		= \Delta v(t,x) + v(t,x) \dot{\fX}(t,x), \
		t>0, \ x \in \bR^d, (d\leq 2)\\
		v(0,x) = 1, \ \dfrac{\partial v}{\partial t} (0,x) = 0
	\end{cases}
\end{align}
More precisely, $\fX=\{\fX(\varphi);\varphi \in \cD(\bR_{+} \times \bR^d)\}$ is a zero-mean Gaussian process with covariance
\[
\bE[\fX(\varphi)\fX(\psi)]=\int_{\bR_{+}}
\int_{(\bR^d)^2}\gamma(x-y)\varphi(t,x)\psi(t,y)dxdydt=:\langle \varphi,\psi \rangle_{\cH_0},
\]
We let $\cH_0$ be the completion of $\cD(\bR_{+} \times \bR^d)$ with respect to the inner product $\langle \varphi,\psi \rangle_0$. Then $\cH_0$ is isomorphic to $L^2(\bR_{+};\cP_0)$.
If (D) holds, equation \eqref{ham1} has a unique solution which has the chaos expansion:
\[
v(t,x)=1+\sum_{n\geq 1}I_n^{\fX}(f_n(\cdot,t,x))
\]
where $I_n^{\fX}$ is the multiple integral with respect to $\fX$ and the kernel $f_n(\cdot,t,x)$ is given by
\[
f_n(t_1,x_1,\ldots,t_n,x_n,t,x)=G_{t-t_n}(x-x_n)\ldots G_{t_2-t_1}(x_2-x_1)1_{\{0<t_1<\ldots<t_n<t\}}.
\]

It is not difficult to see that  $\bE|v(t,x)|^2=\sum_{n\geq 0} J_n(t)$, where
$J_n(t)$ is given by \eqref{def-Jn} for $n\geq 1$, and $J_0(t)=1$.
Using \eqref{rough-bound} and \eqref{bound-fn}, we obtain that for any $t \in [0,1]$ and $x \in \bR^d$,
\[
\bE|u(t,x)|^2=\sum_{n\geq 0} n!\|\widetilde{f}_n(\cdot,x;t)\|_{\cP_0^{\otimes n}}^2 \leq \sum_{n\geq 0} n!  \|f_n(\cdot,x;t)\|_{\cP_0^{\otimes n}}^2 \leq \sum_{n\geq 0}  t^n J_n(t)
 \leq \bE|v(t,x)|^2.
\]
}
\end{remark}

\begin{remark}[Covariance of solution]
\label{rem-cov}
{\rm
For $t>0$, $s>0$, $x \in \bR^d$ and $y \in \bR^d$, using the Wiener chaos decomposition \eqref{eq-chaos expansion}, we see that
\begin{equation}
\label{def-rho}
\rho_{t,s}(x-y)=\bE \left[ \big(u(t,x)-1\big) \big(u(s,y)-1\big) \right]
	= \sum_{n\geq 1}\frac{1}{n!}\alpha_n(x-y;t,s)
\end{equation}
where
\begin{align}
\label{eq-inner product}
\nonumber
\alpha_n(x-y;t,s)=	& (n!)^2 \langle \widetilde f_n(\cdot,x;t), \widetilde f_n(\cdot,y;s) \rangle_{\cP_0^{\otimes n}}=(n!)^2 \langle f_n(\cdot,x;t), \widetilde f_n(\cdot,y;s) \rangle_{\cP_0^{\otimes n}}  \\
\nonumber
	=& n! \sum_{\rho \in S_n} \int_{(\bR^d)^n} \int_{(\bR^d)^n} f_n(x_1, \ldots, x_n,x;t) f_n(y_{\rho(1)}, \ldots, y_{\rho(n)},y;s) \prod_{i=1}^{n}\gamma(x_i-y_i) d\pmb{x} d\pmb{y}\\
\nonumber
	=& n!\sum_{\rho \in S_n} \int_{(\bR^d)^n}
	\cF f_n(\cdot,x;t)(\xi_1,\ldots,\xi_n) \overline{\cF f_n(\cdot,y;s)(\xi_{\rho(1)},\ldots,\xi_{\rho(n)})}
	\mu(d\xi_1) \ldots \mu(d\xi_n) \\
\nonumber
	=& n!\sum_{\rho \in S_n} \int_{(\bR^d)^n}
	e^{-i(\xi_1+\ldots+\xi_n)\cdot (x-y)} \left(\int_{T_n(t)} \prod_{j=1}^n \cF G_{t_{j+1}-t_j}(\xi_1+\ldots+\xi_j) d \pmb{t}\right)\\
	&\left(\int_{T_n(s)} \prod_{j=1}^n \cF G_{s_{j+1}-s_j}(\xi_{\rho(1)}+\ldots+\xi_{\rho(j)}) d \pmb{s} \right)\mu(d\xi_1) \ldots \mu(d\xi_n),
\end{align}
and we denote $t_{n+1}=t$, $s_{n+1}=s$, $\pmb{t}=(t_1,\ldots,t_n)$ and $\pmb{s}=(s_1,\ldots,s_n)$, $\pmb{x}=(x_1,\ldots,x_n)$ and $\pmb{y}=(y_1,\ldots,y_n)$.
In particular, this shows that
$\alpha_n(x-y;t,s)$ and $\rho_{t,s}(x-y)$ are non-negative and depend on $x$ and $y$ only through the difference $x-y$.
}
\end{remark}

\section{Estimate for Malliavin derivative}
\label{sec:Malliavin derivative}


In this section, we will prove some key estimates for the moments of the first and second Malliavin derivatives of the solution to equation \eqref{ham2}.

We will show that for any $z \in \bR^d$ fixed,
\begin{equation}
\label{D-series2}
D_{z}u(t,x)=\sum_{n\geq 1}n I_{n-1}(\widetilde{f}_n(\cdot,z,x;t)):=\sum_{n\geq 1}A_n(z,x;t) \quad \mbox{in $L^2(\Omega)$}.
\end{equation}
More importantly, we will give an estimate for the $p$-th moment of $D_{z}u(t,x)$ showing that the first term of the series above:
\[
A_1(z,x;t)=f_1(z,x;t)=\int_{0}^{t}G_{t-r}(x-z)dr
\]
dominates the other terms.

\medskip

First, note that for any $z \in \bR^d$ fixed, we have the decomposition:
\begin{equation}
\label{decomp}
\widetilde{f}_n(\cdot,z,x;t)=\frac{1}{n}\sum_{j=1}^{n}h_j^{(n)}(\cdot,z,x;t),
\end{equation}
where $h_j^{(n)}(\cdot,z,x;t)$ is the symmetrization of the function $f_j^{(n)}(\cdot,z,x;t)$ given by:
\begin{align*}
&f_j^{(n)}(x_1,\ldots,x_{n-1},z,x;t)=
f_n(x_1,\ldots,x_{j-1},z,x_{j},\ldots,x_{n-1},x;t)\\
&\quad =\int_{\{0<t_1<\ldots<t_{j-1}<r<t_j<\ldots<t_{n-1}<t\}}
G_{t-t_{n-1}}(x-x_{n-1})\ldots G_{t_j-r}(x_j-z)G_{r-t_{j-1}}(z-x_{j-1})\ldots \\
& \qquad \qquad \qquad \qquad \qquad G_{t_2-t_1}(x_2-x_1)dt_1 \ldots dt_{n-1}dr.
\end{align*}

The next result will play an important role in our developments.

\begin{theorem}
\label{main-th1}
Assume that $\gamma$ is non-negative and non-negative definite and the spectral measure $\mu$ satisfies (D), or the noise is white and $d=1$.
If $d=2$, suppose that Hypothesis {\bf (H1)} holds.
Then for any $t>0$, $x\in \bR^d$, $z \in \bR^d$ and $p\geq 2$,
\begin{equation}
\label{D-bound}
\|D_{z}u(t,x)\|_p \leq C \int_0^t G_{t-r}(x-z)dr,
\end{equation}
where the constant $C$ depends on $(p,t,\gamma)$ and is increasing in $t$.
\end{theorem}

\begin{proof}
We first consider the case $p=2$.
We will prove that the series \eqref{D-series2} converges in $L^2(\Omega)$, and therefore,
\begin{equation}
\label{D-series3}
\bE|D_{z}u(t,x)|^2=\sum_{n\geq 1}\bE|A_n(z,x;t)|^2.
\end{equation}
We need to evaluate $\bE|A_n(z,x;t)|^2$. First, note that
\[
A_n(z,x;t)=nI_{n-1}(\widetilde{f}_n(\cdot,z,x;t))=
\sum_{j=1}^{n}I_{n-1}(h_j^{(n)}(\cdot,z,x;t)).
\]

The calculation below will show, in particular, that $f_j^{(n)}(\cdot,z,x;t) \in \cP_0^{\otimes (n-1)}$.

Using the estimate $|\sum_{j=1}^n a_j|^2 \leq n \sum_{j=1}^n |a_j|^2$ and the fact that $\bE|I_n(f)|^2=n! \|\widetilde{f}\|_{\cP_0^{\otimes n}}^2 \leq n! \|f\|_{\cP_0^{\otimes n}}^2 $ for any $f\in \cP_0^{\otimes n}$, we obtain that:
\begin{equation}
\label{e-An}
\bE|A_n(z,x;t)|^2\leq n\sum_{j=1}^{n}\bE|I_{n-1}(h_j^{(n)}(\cdot,z,x;t))|^2 \leq
n\sum_{j=1}^{n}(n-1)!\|f_j^{(n)}(\cdot,z,x;t)\|_{\cP_0^{\otimes (n-1)}}^2.
\end{equation}

The last term of this sum turns out to be easy to estimate. By definition,
\[
f_n^{(n)} (x_1,\ldots,x_{n-1},z,x;t)=\int_0^t G_{t-r}(x-z) f_{n-1}(x_1,\ldots,x_{n-1},z;r)dr.
\]
By the Minkowski's inequality for integrals,
\[
\|f_n^{(n)} (\cdot,z,x;t)\|_{\cP_0^{\otimes (n-1)}} \leq \int_0^t G_{t-r}(x-z) \|f_{n-1}(\cdot,z;r)\|_{\cP_0^{\otimes (n-1)}}dr.
\]
Using relation \eqref{e-fn} and the fact that the constant $D_t$ is increasing in $t$, we infer that
\begin{equation}
\label{j-equal-n}
\|f_n^{(n)}(\cdot,z,x;t)\|_{\cP_0^{\otimes (n-1)}} \leq \frac{t^{n-1}}{(n-1)!} (D_tC_{\mu})^{(n-1)/2} \int_0^t G_{t-r}(x-z)dr.
\end{equation}
Relation \eqref{j-equal-n} can also be deduced from relations \eqref{eq-estimate d=1} and \eqref{norm-fjn} below, with the convention that the second term on the right-hand side is equal to $1$ when $j=n$.

We will prove below that for any $j = 1, \ldots, n-1$,
\begin{align}
\label{eq-1}
	\|f_j^{(n)}(\cdot,z,x;t)\|_{\cP_0^{\otimes (n-1)}}^2 \leq C^n \frac{1}{[(j-1)!]^{2} [(n-j+1)!]^{2}} \left(\int_0^t G_{t-r}(x-z)dr\right)^2,
\end{align}
where $C$ is a constant that depends on $(q,\gamma,t)$, which is increasing in $t$ and may be different from line to line. This inequality holds also for $j=n$, due to \eqref{j-equal-n}. In particular, this inequality implies that $f_j^{(n)}(\cdot,z,x;t) \in \cP_0^{\otimes (n-1)}$ for any $j=1,\ldots,n$. Therefore,
\begin{align*}
	& \sum_{j=1}^{n} \|f_j^{(n)}(\cdot,z,x;t)\|_{\cP_0^{\otimes (n-1)}}^2  \leq
	\frac{C^n}{(n!)^2} \left(\int_0^t G_{t-r}(x-z)dr\right)^2 \sum_{j=1}^{n} \frac{(n!)^2}{[(j-1)!]^{2} [(n-j+1)!]^{2} } \\
	& \quad \quad \quad \leq  \frac{C^n}{(n!)^2} \left(\int_0^t G_{t-r}(x-z)dr\right)^2 \binom{2n}{n} \leq \frac{C^n}{(n!)^2} \left(\int_0^t G_{t-r}(x-z)dr\right)^2,
\end{align*}
where for the second last inequality we used the identity $\sum_{k=0}^n \binom{n}{k}^2=\binom{2n}{n}$, and for the last inequality we used the fact that $(2n)! \leq C^n(n!)^2$, due to Stirling's formula.

Finally, returning to \eqref{e-An}, we have:
\[
\bE|A_n(z,x;t)|^2 \leq n! \sum_{j=1}^{n} \|f_j^{(n)}(\cdot,z,x;t)\|_{\cP_0^{\otimes (n-1)}}^2 \leq \frac{C^n}{n!} \left(\int_0^t G_{t-r}(x-z)dr\right)^2.
\]

This shows that the series \eqref{D-series2} converges in $L^2(\Omega)$, and concludes the proof of \eqref{D-bound} in the case $p=2$.
The case $p>2$ follows by applying Minkowski inequality in $L^p(\Omega)$, and the hypercontractivity property
$\|A_n(z,x;t)\|_p \leq (p-1)^{n/2} \|A_n(z,x;t)\|_2$.

\medskip

It remains to prove \eqref{eq-1} for any $j=1,\ldots,n-1$. We will use the decomposition:
\begin{align}
\nonumber
	f_j^{(n)}(x_1,\ldots,x_{n-1},z,x;t)& =
	\int_{\{0<t_1<\ldots<t_{j-1}<r<t_j<\ldots<t_{n-1}<t\}}
	g_{n-j}(t_j,x_j,\ldots,t_{n-1},x_{n-1},r,z,t,x) \\
\label{f-j-n}
	& \quad \quad \quad f_{j-1}
	(t_1,x_1,\ldots,t_{j-1},x_{j-1},r,z)dt_1 \ldots dt_{n-1}dr
\end{align}
where
\begin{equation}
\label{def-g-nj}
g_{k}(t_1,x_1,\ldots,t_{k},x_{k},r,z,t,x) = G_{t-t_{k}}(x-x_{k}) \ldots G_{t_1-r}(x_1-z).
\end{equation}

We consider separately the cases $d=1$ and $d=2$.

\medskip

{\bf Case $d=1$.}  By \eqref{f-j-n} and Minkowski's inequality for integrals,
\begin{align}
\label{eq-estimate d=1}
	& \|f_j^{(n)}(\cdot,z,x;t)\|_{\cP_0^{\otimes (n-1)}}
	\leq \int_0^t
	\left(\int_{\{0<t_1<\ldots<t_{j-1}<r\}} \|f_{j-1}
	(t_1,\cdot,\ldots,t_{j-1},\cdot,r,z)\|_{\cP_0^{\otimes (j-1)}}dt_1 \ldots dt_{j-1}\right) \nonumber \\
	& \quad \left(\int_{\{r<t_j<\ldots<t_{n-1}<t\}}	\|g_{n-j}(t_j,\cdot,\ldots,t_{n-1},\cdot,r,z,t,x)\|_{\cP_0^{\otimes (n-j)}} dt_j \ldots dt_{n-1} \right)dr.
\end{align}
For each $r \in (0,t)$ fixed, we estimate separately the two integrals above.
Recall that $\cH_0 = L^2(\bR_{+};\cP_0)$. For the first integral, we use the Cauchy-Schwarz inequality:
\begin{align} \label{eq-estimate 1st term d=1}
	&\quad \left(\int_{\{0<t_1<\ldots<t_{j-1}<r\}} \|f_{j-1}
	(t_1,\cdot,\ldots,t_{j-1},\cdot,r,z)\|_{\cP_0^{\otimes (j-1)}}dt_1 \ldots dt_{j-1}\right)^2 \nonumber \\
	\le& \dfrac{r^{j-1}}{(j-1)!}
	\left(\int_{\{0<t_1<\ldots<t_{j-1}<r\}} \|f_{j-1}
	(t_1,\cdot,\ldots,t_{j-1},\cdot,r,z)\|_{\cP_0^{\otimes (j-1)}}^2 dt_1 \ldots dt_{j-1}\right) \nonumber \\
	=& \dfrac{r^{j-1}}{(j-1)!} \|f_{j-1}
	(\cdot,r,z)\|_{\cH_0^{\otimes (j-1)}}^2
    \le \dfrac{C^{j-1}}{[(j-1)!]^2},
\end{align}
where for the last inequality,we used (3.15) of \cite{BNQZ}
(which holds also for the white noise).
For the second integral, we use the fact that $G_t(x)=\frac{1}{2}1_{\{|x|<t\}}$ and
\begin{align*}
	\{|x-x_{n-1}| < t-t_{n-1}\} \cap \ldots \cap \{|x_j-z| < t_j-r\}
	\subset \{|x-z| < t-r\}.
\end{align*}
We have:
\begin{align*}
	& g_{n-j}(t_j,x_j,\ldots,t_{n-1},x_{n-1},r,z,t,x)
	= \dfrac{1}{2^{n-j+1}} \mathbf{1}_{\{|x-x_{n-1}| < t-t_{n-1}\}} \ldots \mathbf{1}_{\{|x_j-z| < t_j-r\}} \\
	& \quad = \dfrac{1}{2^{n-j}} 1_{\{|x-x_{n-1}| < t-t_{n-1}\}} \ldots \mathbf{1}_{\{|x_j-z| < t_j-r\}} \times \dfrac{1}{2} \mathbf{1}_{\{|x-z| < t-r\}} \\
	& \quad \leq \dfrac{1}{2^{n-j}} \mathbf{1}_{\{|x-x_{n-1}| < t-t_{n-1}\}} \ldots \mathbf{1}_{\{|x_{j+1}-x_j| < t_{j+1}-t_j\}} \times \dfrac{1}{2} \mathbf{1}_{\{|x-z| < t-r\}} \\
	& \quad =G_{t-t_{n-1}}(x-x_{n-1}) \ldots G_{t_{j+1}-t_j}(x_{j+1}-x_j) \times G_{t-r}(x-z) \\
	& \quad = f_{n-j} (t_j,x_j,\ldots,t_{n-1},x_{n-1},t,x) G_{t-r}(x-z).
\end{align*}
Applying again the Cauchy-Schwarz inequality, we obtain:
\begin{align} \label{eq-estimate 2nd term d=1}
	& \left( \int_{\{r<t_j<\ldots<t_{n-1}<t\}}	\|g_{n-j}(t_j,\cdot,\ldots,t_{n-1},\cdot,r,z,t,x)\|_{\cP_0^{\otimes (n-j)}} dt_j \ldots dt_{n-1} \right)^2 \nonumber \\
	\le& \dfrac{(t-r)^{n-j}}{(n-j)!} \int_{\{r<t_j<\ldots<t_{n-1}<t\}}	\|g_{n-j}(t_j,\cdot,\ldots,t_{n-1},\cdot,r,z,t,x)\|_{\cP_0^{\otimes (n-j)}}^2 dt_j \ldots dt_{n-1} \nonumber \\
	\le& \dfrac{(t-r)^{n-j}}{(n-j)!} G_{t-r}^2(x-z) \int_{\{r<t_j<\ldots<t_{n-1}<t\}} \|f_{n-j} (t_j,\cdot,\ldots,t_{n-1},\cdot,t,x) \|_{\cP_0^{\otimes (n-j)}}^2 dt_j \ldots dt_{n-1} \nonumber \\
	=& \dfrac{(t-r)^{n-j}}{(n-j)!} G_{t-r}^2(x-z) \int_{\{0<s_j<\ldots<s_{n-1}<t-r\}} \|f_{n-j} (s_j,\cdot,\ldots,s_{n-1},\cdot,t-r,x) \|_{\cP_0^{\otimes (n-j)}}^2 ds_j \ldots ds_{n-1} \nonumber \\
	=& \dfrac{(t-r)^{n-j}}{(n-j)!} G_{t-r}^2(x-z) \|f_{n-j}
	(\cdot,t-r,z)\|_{\cH_0^{\otimes (n-j)}}^2
	\le \dfrac{(t-r)^{n-j} C(t-r)^{n-j}}{[(n-j)!]^2} G_{t-r}^2(x-z),
\end{align}
where for the fourth line we used the change of variables $s_j=t_j-r$, and for the last line we used (3.15) of \cite{BNQZ}.
Relation \eqref{eq-1} follows substituting \eqref{eq-estimate 1st term d=1} and \eqref{eq-estimate 2nd term d=1} into \eqref{eq-estimate d=1}.

\medskip

{\bf Case $d=2$.}
Using \eqref{embed}, it suffices to estimate
$\| f_j^{(n)}(\cdot,z,x;t)\|_{L^{2q}(\bR^{2(n-1)})}^2$. 
We use decomposition \eqref{f-j-n}.
By Minkowski's inequality for integrals, 
\begin{align}
\nonumber
& \|f_j^{(n)}(\cdot,z,x;t)\|_{L^{2q}(\bR^{2(n-1)})} \leq \\
\nonumber
& \quad \int_0^t
\left(\int_{\{0<t_1<\ldots<t_{j-1}<r\}} \|f_{j-1}
(t_1,\cdot,\ldots,t_{j-1},\cdot,r,z)\|_{L^{2q}(\bR^{2(j-1)})}dt_1 \ldots dt_{j-1}\right) \\
\label{norm-fjn}
& \quad \left(\int_{\{r<t_j<\ldots<t_{n-1}<t\}}
\|g_{n-j}(t_j,\cdot,\ldots,t_{n-1},\cdot,r,z,t,x)\|_{L^{2q}(\bR^{2(n-j)})}
 dt_j \ldots dt_{n-1} \right)dr
\end{align}
For any $r\in (0,t)$ fixed, we estimate separately the two integrals.
Since $q<1$, $G_t^{2q}$ is integrable (see \eqref{p-norm-G}). Hence, for the first integral, we have:
\begin{align}
\nonumber
& \int_{\{0<t_1<\ldots<t_{j-1}<r\}} \|f_{j-1}
(t_1,\cdot,\ldots,t_{j-1},\cdot,r,z)\|_{L^{2q}(\bR^{2(j-1)})}dt_1 \ldots dt_{j-1} \\
\nonumber
& \quad = C^{j-1} \int_{\{0<t_1<\ldots<t_{j-1}<r\}} (r-t_{j-1})^{\frac{1-q}{q}} \ldots (t_2-t_1)^{\frac{1-q}{q}}dt_1 \ldots dt_{j-1}\\
\label{f-j-1}
& \quad =C^{j-1}\frac{\Gamma(1/q)^{j-1}r^{(j-1)/q}}{\Gamma((j-1)/q+1)} \leq \frac{C^{j-1}}{[(j-1)!]^{1/q}} t^{(j-1)/q} \leq \frac{C^{j-1}}{(j-1)!} t^{(j-1)/q},
\end{align}
where we used Stirling's formula for the second last inequality. As for the second integral, by the Cauchy-Schwarz inequality, we have:
\begin{align*}
& [I_{z,t,x}^{(j,n)}(r)]^2 :=\left(\int_{\{r<t_j<\ldots<t_{n-1}<t\}}
\|g_{n-j}(t_j,\cdot,\ldots,t_{n-1},\cdot,r,z,t,x)\|_{L^{2q}(\bR^{2(n-j)})}
 dt_j \ldots dt_{n-1}\right)^2  \\
& \quad \leq \frac{(t-r)^{n-j}}{(n-j)!} \int_{\{r<t_j<\ldots<t_{n-1}<t\}}
\|g_{n-j}(t_j,\cdot,\ldots,t_{n-1},\cdot,r,z,t,x)\|_{L^{2q}(\bR^{2(n-j)})}^2
 dt_j \ldots dt_{n-1} \\
& \quad = : \frac{(t-r)^{n-j}}{(n-j)!} \cT_{n-j+1}(r,z,t,x).
\end{align*}
By relations (3.18), (3.20), (3.22) and (3.23) of \cite{BNQZ}, we know that
\begin{align*}
& \cT_{n-j+1}(r,z,t,x)\\
&=\int_{\{r<t_j<\ldots<t_{n-1}<t\}}
\left(\int_{\bR^{2(n-j)}} G_{t-t_{n-1}}^{2q}(x-x_{n-1})\ldots G_{t_j-r}^{2q}(x_j-z) dx_1 \ldots dx_j \right)^{1/q} dt_j \ldots dt_{n-1}\\
& \leq
\left\{
\begin{array}{ll}
C G_{t-r}^{2-1/q}(x-r)  & \mbox{if $n-j+1=2,3,4$} \\
C^{n-j+1} \frac{1}{(n-j+2)!} 1_{\{|x-z|<t-r\}} & \mbox{if $n-j+1 \geq 5$}.
\end{array} \right.
\end{align*}
Hence,
\begin{align*}
I_{z,t,x}^{(j,n)}(r) \leq
\left\{
\begin{array}{ll}
C G_{t-r}^{1-1/(2q)}(x-r)  & \mbox{if $n-j=1,2,3$} \\
C^{n-j+1} \frac{1}{[(n-j)!]^{1/2}} \frac{1}{[(n-j+2)!]^{1/2}} 1_{\{|x-z|<t-r\}} & \mbox{if $n-j \geq 4$},
\end{array} \right.
\end{align*}
Using properties \eqref{G-monotone} and \eqref{I-less-G} of $G$, we obtain that:
\begin{equation}
\label{bound-I-ztx}
I_{z,t,x}^{(j,n)}(r) \leq \frac{C^{n-j+1}}{(n-j+1)!} G_{t-r}(x-z) \quad \mbox{if $n-j\geq 1$}.
\end{equation}
Coming back to \eqref{norm-fjn}, and using \eqref{f-j-1} and \eqref{bound-I-ztx}, we obtain that for any $j=1,\ldots,n-1$,
\begin{equation}
\label{bound-fjn}
\|f_j^{(n)}(\cdot,z,x;t)\|_{L^{2q}(\bR^{2(n-1)})} \leq C^n \frac{1}{(j-1)! (n-j+1)!} \int_0^t G_{t-r}(x-z)dr.
\end{equation}
By \eqref{embed}, this concludes the proof of \eqref{eq-1} in the case $d=2$.
\end{proof}

We end this section with a similar estimate for the second Malliavin derivative.

\begin{theorem}
\label{main-th2}
Under the hypotheses of Theorem \ref{main-th1}, for any $t>0$, $x,w,z \in \bR^d$ and $p\geq 2$,
\begin{equation}
\label{bound-D2}
\|D_{w,z}^2 u(t,x) \|_p \leq C \widetilde{f}_2(w,z,x;t),
\end{equation}
where $C>0$ is a constant that depends on $(p,t,\gamma)$ and is increasing in $t$.
\end{theorem}

\begin{proof}
{\em Step 1.} We first prove that for any $w,z \in \bR^d$ fixed, the following series converges in $L^2(\Omega)$:
\begin{equation}
\label{D2-series}
D_{w,z}^2 u(t,x)=\sum_{n\geq 2}n(n-1)I_{n-2}(\widetilde{f}_n(\cdot,w,z,x;t))=: \sum_{n\geq 2}B_n(w,z,x;t)
\end{equation}
We need to evaluate $\bE|B_n(w,z,x;t)|^2$. The first term of this series is
\begin{align*}
& B_2(w,z,x;t)=2 \widetilde{f}_2(w,z,x;t)=f_2(w,z,x;t)+f_2(z,w,x;t) \\
& \quad =\int_{0<\theta<r<t}G_{t-r}(x-z)G_{r-\theta}(z-w)drd\theta+
\int_{0<r<\theta<t}G_{t-\theta}(x-w)G_{\theta-r}(w-z)drd\theta.
\end{align*}

 Note that we have the following decomposition:
 \begin{equation}
 \label{f-ij-n}
 \widetilde{f}_n(\cdot,w,z,x;t)=\frac{1}{n(n-1)} \sum_{i,j=1,i\not=j}^{n} h_{ij}^{(n)}(\cdot,w,z,x;t),
 \end{equation}
 where $h_{ij}^{(n)}(\cdot,w,z,x;t)$ is the symmetrization of the function $f_{ij}^{(n)}(\cdot,w,z,x;t)$ defined as follows. If $i<j$,
\begin{align*}
& f_{ij}^{(n)}(x_1,\ldots,x_{n-2},w,z,x;t)=f_n(x_1,\ldots,x_{i-1},w,x_{i},
\ldots,x_{j-2},z,x_{j-1},\ldots,x_{n-2},x;t) \\
& =
\int_{\{t_1<\ldots<t_{i=1}<\theta<t_i<\ldots<t_{j-2}<r<t_{j-1}<\ldots<
t_{n-2}<t\}}
G_{t-t_{n-2}}(x-x_{n-2}) \ldots G_{t_{j-1}-r}(x_{j-1}-z) \\
& \quad G_{r-t_{j-2}}(z-x_{j-2}) \ldots G_{t_i-\theta}(x_i-w)G_{\theta-t_{i-1}}(w-x_{i-1}) \ldots G_{t_2-t_1}(x_2-x_1) dt_1 \ldots dt_{n-2}drd\theta.
 \end{align*}
If $j<i$,
 \[
 f_{ij}^{(n)}(x_1,\ldots,x_{n-2},w,z,x;t)=f_n(x_1,\ldots,x_{j-1},z,x_{j},
\ldots,x_{i-2},w,x_{i-1},\ldots,x_{n-2},x;t)
\]
 In both cases, $w$ is on position $i$ and $z$ is on position $j$.

Hence,
\[
B_n(w,z,x;t)=n(n-1) I_{n-2}(\widetilde{f}_n(\cdot,w,z,x;t))=\sum_{i,j=1,i\not=j}^n I_{n-2}(h_{ij}^{(n)}(\cdot,w,z,x;t))
\]
and
\begin{align*}
\bE|B_n(w,z,x;t)|^2 & \leq n(n-1) \sum_{i,j=1,i\not=j}^n \bE|I_{n-2}(h_{ij}^{(n)}(\cdot,w,z,x;t))|^2 \\
&=n(n-1)\sum_{i,j=1,i\not=j}^n (n-2)! \|h_{ij}^{(n)}(\cdot,w,z,x;t))\|_{\cP_0^{\otimes (n-2)}}^2\\
& \leq n(n-1)\sum_{i,j=1,i\not=j}^n (n-2)! \|f_{ij}^{(n)}(\cdot,w,z,x;t))\|_{\cP_0^{\otimes (n-2)}}^2
\end{align*}

We will prove below that
\begin{align}
\label{f-ij-n1}
\|f_{ij}^{(n)}(\cdot,w,z,x;t)\|_{\cP_0^{\otimes (n-2)}} \leq \frac{C^n}{(i-1)!(j-i)!(n-j+1)!}f_2(w,z,x;t) \quad \mbox{if} \quad i<j \\
\label{f-ij-n2}
\|f_{ij}^{(n)}(\cdot,w,z,x;t)\|_{\cP_0^{\otimes (n-2)}} \leq \frac{C^n}{(j-1)!(i-j)!(n-i+1)!}f_2(z,w,x;t) \quad \mbox{if} \quad j<i.
\end{align}

Then
\begin{align*}
\bE|B_n(w,z,x;t)|^2 & \leq n! \sum_{i,j=1,i<j}^{n}\frac{C^n}{[(i-1)! (j-i)!(n-j+1)!]^2} f_2^2 (w,z,x;t)+ \\
& \quad n! \sum_{i,j=1,j<i}^{n}\frac{C^n}{[(j-1)! (i-j)!(n-i+1)!]^2} f_2^2 (z,w,x;t)\\
& \leq \frac{C^n}{n!} ( f_2^2 (w,z,x;t)+ f_2^2 (z,w,x;t)),
\end{align*}
using the fact that:
\begin{align*}
	\sum_{k_1+k_2+k_3=n} \binom{n}{k_1, k_2, k_3}^2
	\le \left( \sum_{k_1+k_2+k_3=n} \binom{n}{k_1, k_2, k_3} \right)^2
	= 9^n
\end{align*}

This proves that the series \eqref{D2-series} converges in $L^2(\Omega)$.

\medskip

{\em Step 2.} We prove \eqref{bound-D2}. By hypercontractivity and the decomposition \eqref{f-ij-n},
\begin{align*}
\|D_{w,z}^{2}u(t,x)\|_p & \leq \sum_{n\geq 2} (p-1)^{n/2} [(n-2)!]^{1/2} \sum_{i,j=1,i<j}^{n}\|h_{ij}^{(n)}(\cdot,w,z,x;t)\|_{\cP_0^{\otimes (n-2)}}\\
&  \leq \sum_{n\geq 2} (p-1)^{n/2} [(n-2)!]^{1/2} \sum_{i,j=1,i<j}^{n}\|f_{ij}^{(n)}(\cdot,w,z,x;t)\|_{\cP_0^{\otimes (n-2)}}.
\end{align*}
Relation \eqref{bound-D2} follows using \eqref{f-ij-n1} and \eqref{f-ij-n2} and the fact that
\[
\sum_{k_1+k_2+k_3=n} \left(\begin{array}{c}n \\ k_1,k_2,k_3 \end{array}\right) = 3^n.
\]

\medskip

{\em Step 3.} It remains to prove \eqref{f-ij-n1} (the proof of \eqref{f-ij-n2} is similar).
For $i<j$, we write
\begin{align*}
	& f_{ij}^{(n)}(x_1,\ldots,x_{n-2},w,z,x;t) \\
	=& \int_{\{0<t_1<\ldots<t_{i-1}<\theta<t_i<\ldots< t_{j-2}<r<t_{j-1}<\ldots<t_{n-2}<t\}} g_{n-j}(t_{j-1},x_{j-1},\ldots,t_{n-2},x_{n-2},r,z,t,x) \\
	& g_{j-i-1}(t_{i},x_{i},\ldots,t_{j-2},x_{j-2},\theta,w,r,z) f_{i-1} (t_1,x_1,\ldots,t_{i-1},x_{i-1},\theta,w) dt_1 \ldots dt_{n-2}drd\theta,
\end{align*}
where the function $g_{k}$ is given by \eqref{def-g-nj}.

For the case $d=1$, we apply the Minkowski's inequality for the norm $\|\cdot\|_{\cP_0^{\otimes (n-2)}}$ and the separation of variables to obtain
\begin{align*}
	& \big\| f_{ij}^{(n)}(\cdot,w,z,x;t) \big\|_{\cP_0^{\otimes (n-2)}} \\
	\leq & \int_{0<\theta<r<t} drd\theta
	\left(\int_{\{0<t_1<\ldots<t_{i-1}<\theta\}} \big\| f_{i-1} (t_1,\cdot,\ldots,t_{i-1},\cdot,\theta,w) \big\|_{\cP_0^{\otimes (i-1)}} dt_1\ldots dt_{i-1}\right) \\
	& \left(\int_{\{\theta<t_i<\ldots< t_{j-2}<r\}} \big\| g_{j-i-1}(t_{i},\cdot,\ldots,t_{j-2},\cdot,\theta,w,r,z) \big\|_{\cP_0^{\otimes (j-i-1)}} dt_i \ldots dt_{j-2}\right) \\
	& \left(\int_{\{r<t_{j-1}<\ldots<t_{n-2}<t\}} \big\| g_{n-j}(t_{j-1},\cdot,\ldots,t_{n-2},\cdot,r,z,t,x) \big\|_{\cP_0^{\otimes (n-j)}} dt_{j-1} \ldots dt_{n-2}\right) \\
	\le& \int_{0<\theta<r<t} \dfrac{C^{i-1}}{(i-1)!} \times \dfrac{C^{j-i-1}}{(j-i-1)!} G_{r-\theta}(z-w) \times \dfrac{C^{n-j}}{(n-j)!} G_{t-r}(x-z) drd\theta \\
	=& \dfrac{C^n}{(i-1)!(j-i-1)!(n-j)!} f_2(w,z,x;t),
\end{align*}
where we use \eqref{eq-estimate 1st term d=1} and \eqref{eq-estimate 2nd term d=1} in the last inequality.

For the case $d=2$, one can compute the norm $\|\cdot\|_{L^{2q}(\bR^{2(n-2)})}$ first. Applying the Minkowski's inequality and the separation of variables, we have
\begin{align*}
	& \| f_{ij}^{(n)}(\cdot,w,z,x;t) \|_{L^{2q}(\bR^{2(n-2)})} \\
	\leq & \int_{\{0<\theta<r<t\}} drd\theta
	\left(\int_{\{0<t_1<\ldots<t_{i-1}<\theta\}} \big\| f_{i-1} (t_1,\cdot,\ldots,t_{i-1},\cdot,\theta,w) \big\|_{L^{2q}(\bR^{2(i-1)})} dt_1\ldots dt_{i-1} \right)\\
	& \left(\int_{\{\theta<t_i<\ldots<t_{j-2}<r\}} \big\| g_{j-i-1}(t_{i},\cdot,\ldots,t_{j-2},\cdot,\theta,w,r,z) \big\|_{L^{2q}(\bR^{2(j-i-1)})} dt_i\ldots dt_{j_2} \right)\\
	& \left(\int_{\{r<t_1<\ldots<t_{i-1}<t\}} \big\| g_{n-j}(t_{j-1},\cdot,\ldots,t_{n-2},\cdot,r,z,t,x) \big\|_{L^{2q}(\bR^{2(n-j)})} dt_{j-1}\ldots dt_{n-2}\right) \\
	\le& \int_{\{0<\theta<r<t\}} \dfrac{C^{i-1}}{(i-1)!} t^{(i-1)/q} \times \dfrac{C^{j-i}}{(j-i)!} G_{r-\theta}(z-w) \times \dfrac{C^{n-j+1}}{(n-j+1)!} G_{t-r}(x-z) drd\theta \\
	=& \dfrac{C^n}{(i-1)!(j-i)!(n-j+1)!} f_2(w,z,x;t),
\end{align*}
where we use \eqref{f-j-1} and \eqref{bound-I-ztx} in the last inequality. The desired result \eqref{f-ij-n1} follows directly from \eqref{embed}.

\end{proof}

\begin{remark}
{\rm
Using a similar argument, one can extend \eqref{bound-D2} to higher order Malliavin derivatives $D^m u(t,x)$ for $m\geq 3$, as it was done in \cite{BNQZ} for the colored noise in time.
}
\end{remark}

\section{Proof of Theorem \ref{CLT-integr}}
\label{section-CLT1}

In this section, we give the proof of Theorem \ref{CLT-integr}.
We treat only the case of integrable function $\gamma$. The case of the white noise with $d=1$ is similar, and is omitted.

Since $\gamma$ is integrable, $\mu$ has density function $g$ given by:
\[
g(\xi)=\frac{1}{(2\pi)^d}\int_{\bR^d}e^{i\xi \cdot x}\gamma(x)dx, \quad \xi \in \bR^d.
\]
Note that $g$ is continuous and bounded. More precisely, $\|g\|_{\infty} \leq (2\pi)^{-d} \|\gamma\|_{L^1(\bR^d)}$.

\subsection{Proof of Theorem \ref{CLT-integr}.(i)}
\label{section-proof-i}

Recalling the definition \eqref{def-rho} of the covariance function $\rho_{t,s}$, we have by Fubini theorem,
\[
\frac{\bE [F_R(t)F_R(s)]}{R^d}= \frac{1}{R^d}\int_{B_R^2} \rho_{t,s}(x-y) dxdy =\int_{B_R} \frac{{\rm Leb}(B_R \cap B_R(-z))}{R^d}  \rho_{t,s}(z)dz.
\]
We now intend to apply the dominated convergence theorem, using the fact that
$\frac{{\rm Leb}(B_R \cap B_R(-z))}{R^d}$ converges to $\omega_d$ as $R \to \infty$, and is bounded by $\omega_d$. But to justify the application of this theorem, we have to prove that:
\begin{equation}
\label{cov-summable}
\int_{\bR^d}\rho_{t,s}(z)dz<\infty,
\end{equation}
and for this, {\em we will use the fact that $\gamma$ is integrable}.
(Note that the converse is also true: if \eqref{cov-summable} holds then $\gamma$ is integrable. This is due to the non-negativity of $\alpha_n(z;t,s)$ and relations \eqref{int-rho} and \eqref{ineq-n1} below.)
First, note that
\begin{equation} \label{int-rho}
	\int_{\bR^d} \rho_{t,s}(z)dz
	=\sum_{n\geq 1} \frac{1}{n!} \int_{\bR^d} \alpha_n(z;t,s)dz.
\end{equation}

We consider first the case $n=1$. By direct calculation,
\begin{align} \nonumber
	\int_{\bR^d}\alpha_1(z;t,s)dz&=\int_{\bR^d}\langle f_1(x_1,z;t),f_1(y_1,0;s)\rangle_{\cP_0}dz\\
	\nonumber
	&=\int_0^t \int_0^s \int_{(\bR^d)^3} G_{t-t_1}(z-x_1)G_{s-s_1}(y_1)\gamma(x_1-y_1)dzdx_1 dy_{1} ds_1 dt_1\\
	\label{ineq-n1}
	&=
	\|\gamma\|_{L^1(\bR^d)} \int_0^t \int_0^s(t-t_1)(s-s_1)ds_1 dt_1=\|\gamma\|_{L^1(\bR^d)}\frac{t^2+s^2}{2},
\end{align}
integrating the variables $z,x_1,y_1$ in this order.

Next, we consider the case $n\geq 2$. By the monotone convergence theorem,
\begin{equation}
\label{MCT}
\int_{\bR^d}\alpha_n(z;t,s)dz=\lim_{\e \downarrow 0} \int_{\bR^d}\alpha_n(z;t,s)e^{-\frac{\e|z|^2}{2}}dz.
\end{equation}
We assume that $s\leq t$.
By \eqref{eq-inner product}, we have
\begin{align*}
T_{n,\e}:=& \int_{\bR^d} \alpha_n(z;t,s) e^{-\frac{\e |z|^2}{2}} dz
=  n! (2\pi)^{d} \sum_{\rho \in S_n} \int_{T_n(t)} \int_{T_n(s)} \int_{(\bR^d)^n}
\prod_{j=1}^n \cF G_{t_{j+1}-t_j}(\xi_1+\ldots+\xi_j) \\
& \quad \quad \quad \prod_{j=1}^n \cF G_{s_{j+1}-s_j}(\xi_{\rho(1)}+\ldots+\xi_{\rho(j)})
p_{\e}(\sum_{j=1}^n \xi_{j})
	  \mu(d\xi_1) \ldots \mu(d\xi_n)d\pmb{t} d\pmb{s}\\
\leq &  n! (2\pi)^{d} t^2 \int_{T_n(t)} \int_{(\bR^d)^n}
\prod_{j=1}^n |\cF G_{t_{j+1}-t_j}(\xi_1+\ldots+\xi_j)|^2
p_{\e}(\sum_{j=1}^n \xi_{j})
	  \mu(d\xi_1) \ldots \mu(d\xi_n)d\pmb{t},
\end{align*}
where for the second line we used the fact that $\int_{\bR^d} e^{-i \xi \cdot z} e^{-\frac{\e |z|^2}{2}}  dz=(2\pi)^d p_{\e}(\xi)$ with $p_{\e}(x)=(2\pi\e)^{-d/2} e^{-|x|^2/(2\e)}$, and for the last line we applied Lemma \ref{Lem-inner product inequality} to the measure
$\mu_n(d\xi_1 \ldots d\xi_n) = p_{\e}(\sum_{j=1}^n \xi_{j})\mu(d\xi_1) \ldots \mu(d\xi_n)$.
Using the fact that that $\mu(d\xi)=g(\xi)d\xi$ and the change the variables $\eta_j = \xi_1 + \ldots + \xi_j$ for $j=1,\ldots,n$ (with $\eta_0=0$), we obtain:
\begin{align*}
T_{n,\e}	\leq & n! (2\pi)^{d} t^2 \int_{T_n(t)} \int_{(\bR^d)^{n-1}} \prod_{j=1}^{n-1} \left| \cF G_{t_{j+1}-t_j}(\eta_j) \right|^2
\prod_{j=1}^{n-1} g(\eta_j-\eta_{j-1}) \\
& \quad \left(\int_{\bR^d}\left| \cF G_{t-t_n}(\eta_n) \right|^2  p_{\e}(\eta_n)g(\eta_n-\eta_{n-1})d\eta_n\right)  d\eta_1 \ldots d\eta_{n-1} d\pmb{t}
\end{align*}
For the inner integral, we use the fact that $\|g\|_{\infty}\leq (2\pi)^{-d}\|\gamma \|_{L^1(\bR^d)}$, and so,
\begin{align*}
(2\pi)^{d}\int_{\bR^d}\left| \cF G_{t-t_n}(\eta_n) \right|^2  p_{\e}(\eta_n)\varphi(\eta_n-\eta_{n-1})d\eta_n & \leq  \|\gamma\|_{L^1(\bR^d)} \int_{\bR^d}\frac{\sin^2((t-t_n)|\eta_n|)}{|\eta_n|^2}p_{\e}(\eta_n)d\eta_n\\
 &\leq \|\gamma\|_{L^1(\bR^d)} (t-t_n)^2 \leq \|\gamma\|_{L^1(\bR^d)} t^2.
\end{align*}
Hence
\begin{align*}
T_{n,\e}& \leq n! t^4 \|\gamma\|_{L^{1}(\bR^d)} \int_{T_n(t)}  \int_{(\bR^d)^{n-1}} \prod_{j=1}^{n-1} \left| \cF G_{t_{j+1}-t_j}(\eta_j) \right|^2 \prod_{j=1}^{n-1} g(\eta_j-\eta_{j-1}) d\eta_1 \ldots d\eta_{n-1} d\pmb{t}\\
	&= n! t^4 \|\gamma\|_{L^{1}(\bR^d)} \int_{T_n(t)}\int_{(\bR^d)^{n-1}} \prod_{j=1}^{n-1} \left| \cF G_{t_{j+1}-t_j}(\xi_1 + \ldots + \xi_j) \right|^2  \mu(d\xi_1) \ldots \mu(d\xi_{n-1}) d\pmb{t}\\
& \leq t^4 \|\gamma\|_{L^{1}(\bR^d)} (D_t C_{\mu})^{n-1} t^n,
\end{align*}
where $D_t=2(t^2 \vee 1)$ and $C_{\mu}=\int_{\bR^d}
\frac{1}{1+|\xi|^2}\mu(d\xi)$ (see \eqref{Jn-bound} for the last inequality).
This bound is independent of $\e$. Therefore, by \eqref{MCT}, it follows that for any $n\geq 2$,
\begin{equation}
\label{ineq-n2}
\int_{\bR^d} \alpha_n(z;t,s)dz \leq  t^4 \|\gamma\|_{L^{1}(\bR^d)} (D_t C_{\mu})^{n-1} t^n.
\end{equation}
Coming back to \eqref{int-rho} and using \eqref{ineq-n1} and \eqref{ineq-n2}, we obtain:
\[
\int_{\bR^d}\rho_{t,s}(z)dz \leq t^4 \|\gamma\|_{L^1(\bR^d)} \left( 1+\sum_{n\geq 2}\frac{1}{n!} (D_t C_{\mu})^{n-1} t^n \right)<\infty.
\]
This concludes the proof of \eqref{cov-summable}.

\subsection{Proof of Theorem \ref{CLT-integr}.(ii)}
\label{section-proof-ii}

We apply Proposition 1.8 of \cite{BNQZ}, which continues to hold for the time-independent noise with obvious modifications. We obtain:
\begin{equation}
\label{dTV-bound}
d_{TV}\left( \frac{F_R(t)}{\sigma_R(t)},Z\right)=d_{TV}(F_R(t),N)\leq \frac{4}{\sigma_R^2(t)}\sqrt{\cA}
\end{equation}
where $N \sim N(0,\sigma_R^2(t))$ and
\begin{align*}
\cA & =\int_{(\bR^d)^6} \|D_{z,w}^2 F_R(t)\|_4 \|D_{y,w'}^2 F_R(t)\|_4
\|D_{z'} F_R(t)\|_4 \|D_{y'} F_R(t)\|_4 \\
& \qquad \qquad \qquad \quad \gamma(y-y')\gamma(z-z')\gamma(w-w')dydy'dzdz'dwdw'.
\end{align*}
Since $\sigma_R^2(t) \sim K(t,t) R^d$ (by part (i)), it is enough to prove that
\begin{equation}
\label{bound-A-integr}
\cA\leq C R^d,
\end{equation}
where $C>0$ is a constant that depends on $(t,\gamma,d)$.

By Minkowski's inequality and Theorem \ref{main-th2}, for any $z,w \in \bR^d$,
\begin{align*}
& \|D_{z,w}^2 F_R(t)\|_4 =\left\|\int_{B_R} D_{z,w}^2 u(t,x) dx \right\|_4 \leq \int_{B_R} \|D_{z,w}^2 u(t,x) \|_4 dx \leq C \int_{B_R}\widetilde{f}_2(w,z,x;t)dx= \\
& C\int_{B_R}\left(\int_{0<\theta<r<t} G_{t-r}(x-z)G_{r-\theta}(z-w) dr d\theta  + \int_{0<r<\theta<t}G_{t-\theta}(x-w)G_{\theta-r}(w-z) dr d\theta  \right)dx.
\end{align*}
Similarly, by Minkowski's inequality and Theorem \ref{main-th1}, for any $z \in \bR^d$,
\begin{align*}
& \|D_{z}F_R(t)\|_4 =\left\|\int_{B_R} D_{z}u(t,x) dx \right\|_4 \leq \int_{B_R} \|D_{z} u(t,x) \|_4 dx \leq C \int_{B_R} \int_0^t G_{t-s}(x-z) dsdx.
\end{align*}
It follows that
\begin{equation}
\label{4terms-A}
\cA \leq C \sum_{j=1}^4\cA_{j},
\end{equation}
where
\begin{align*}
\cA_{1}&=\int_{[0,t]^2} \int_{0<\theta<r<t} \int_{0<\theta'<r'<t}
\int_{B_R^4} \int_{(\bR^d)^6} G_{t-r}(x_1-z) G_{r-\theta} (z-w) G_{t-r'}(x_1'-y) \\
& \qquad \qquad \qquad G_{r'-\theta'}(y-w')  G_{t-s}(x_2-z') G_{t-s'}(x_2'-y') \gamma(y-y')\gamma(z-z')\gamma(w-w') \\
&  \qquad \qquad \qquad dydy' dzdz' dwdw' dx_1 dx_2 dx_1' dx_2' dr' d\theta' dr d\theta  dsds'\\
\cA_{2}&=\int_{[0,t]^2} \int_{0<\theta<r<t} \int_{0<r'<\theta'<t}
\int_{B_R^4} \int_{(\bR^d)^6} G_{t-r}(x_1-z) G_{r-\theta} (z-w) G_{t-\theta'}(x_1'-w) \\
& \qquad \qquad \qquad G_{\theta'-r'}(w'-y)  G_{t-s}(x_2-z') G_{t-s'}(x_2'-y') \gamma(y-y')\gamma(z-z')\gamma(w-w') \\
&  \qquad \qquad \qquad dydy' dzdz' dwdw' dx_1 dx_2 dx_1' dx_2' dr' d\theta' dr d\theta  dsds'\\
\cA_{3}&=\int_{[0,t]^2} \int_{0<r<\theta<t} \int_{0<\theta'<r'<t}
\int_{B_R^4} \int_{(\bR^d)^6} G_{t-\theta}(x_1-w) G_{\theta-r} (w-z) G_{t-r'}(x_1'-y)  \\
& \qquad \qquad \qquad G_{r'-\theta'}(y-w') G_{t-s}(x_2-z') G_{t-s'}(x_2'-y') \gamma(y-y')\gamma(z-z')\gamma(w-w') \\
&  \qquad \qquad \qquad dydy' dzdz' dwdw' dx_1 dx_2 dx_1' dx_2' dr' d\theta' dr d\theta  dsds'\\
\cA_{4}&=\int_{[0,t]^2} \int_{0<r<\theta<t} \int_{0<r'<\theta'<t}
\int_{B_R^4} \int_{(\bR^d)^6} G_{t-\theta}(x_1-w) G_{\theta-r} (w-z) G_{t-\theta'}(x_1'-w')  \\
& \qquad \qquad \qquad G_{\theta'-r'}(w'-y) G_{t-s}(x_2-z') G_{t-s'}(x_2'-y') \gamma(y-y')\gamma(z-z')\gamma(w-w') \\
&  \qquad \qquad \qquad dydy' dzdz' dwdw' dx_1 dx_2 dx_1' dx_2' dr' d\theta' dr d\theta  dsds'.
\end{align*}

We treat separately the 4 terms. We start with $\cA_{1}$.
We have 10 integrals in the space variables and the integrand is a product of 9 functions. Using the fact that $\int_{\bR^d}G_t(x)dx=t$ and $\|\gamma\|_{L^1(\bR^d)}<\infty$, we integrate the space variables in the order $x_2',y',x_2,z',x_1',y,w',w,z$, using one function ($G$ or $\gamma$) at a time. The remaining integral $dx_1$ on $B_R$ (for which there is no function $G$ or $\gamma$ to integrate) yields the factor ${\rm Leb}(B_R)=\omega_d R^d$. The remaining iterated integral in the 6 time variables is bounded by $t^6$. We obtain that:
\[
\cA_{1} \leq  t^{12} \|\gamma\|_{L^1(\bR^d)}^3 \omega_d R^d.
\]
A similar argument works for $\cA_2,\cA_3,\cA_4$. For $\cA_2$, we use the same order of integration as for $\cA_1$. For $\cA_3,\cA_4$, we integrate in the order $x_2',y',x_2,z',x_1',y,w',z,w$.
This proves \eqref{bound-A-integr}.

\subsection{Proof of Theorem \ref{CLT-integr}.(iii)}
\label{section-proof-iii}

{\em Step 1.} (tightness) We prove that for any $p\geq 2$ and $0<s<t<T$
\[
 \|F_R(t)-F_R(s)\|_p \leq C R^{d/2}(t-s).
\]
By Kolmogorov's continuity theorem, it will follow that the process
$\{F_R(t)\}_{t\geq 0}$ has a continuous modification.

Using the chaos expansion, we can write
$F_R(t)-F_R(s)= \sum_{n\geq 1} I_n (g_{n,R}(\cdot;t,s))$,
where
\begin{align*}
& g_{n,R}(x_1,\ldots,x_n;t,s)
	= \int_{B_R} \big( f_n(x_1,\ldots,x_n,x;t) - f_n(x_1,\ldots,x_n,x;s) \big)dx \nonumber \\
& \quad =\int_{B_R} \int_{T_n(t)} \prod_{j=1}^{n-1} G_{t_{j+1}-t_j} (x_{j+1} - x_j) \big( G_{t-t_n}(x-x_n) - G_{s-t_n}(x-x_n) \big)  d\pmb{t} dx
\end{align*}
and $\pmb{t}=(t_1,\ldots,t_n)$. Here we used the convention $\prod_{j=1}^0 = 1$ and the fact that
\begin{align*}
	\int_{T_n(s)} G_{s-t_n}(x-x_n) \prod_{j=1}^{n-1} G_{t_{j+1}-t_j} (x_{j+1} - x_j) d\pmb{t}
	= \int_{T_n(t)} G_{s-t_n}(x-x_n) \prod_{j=1}^{n-1} G_{t_{j+1}-t_j} (x_{j+1} - x_j) d\pmb{t}.
\end{align*}
since if $\pmb{t} \in T_n(t) \setminus T_n(s)$, then $t_n > s$ and $G_{s-t_n}(x)=0$, due to our convention \eqref{rule1}.

The Fourier transform in the spatial variables of the kernel $g_{n,R}(\cdot;t,s)$ is
\begin{align}
\label{eq-Fourier of g}
\nonumber
& \cF g_{n,R}(\cdot;t,s)(\xi_1,\ldots,\xi_n)= \int_{T_n(t)} \int_{B_R} e^{-ix \cdot (\xi_1 + \ldots + \xi_n)} \prod_{j=1}^{n-1} \cF G_{t_{j+1}-t_j} (\xi_1 + \ldots + \xi_j)\\
& \qquad \qquad \qquad \qquad
	\Big( \cF G_{t-t_n} (\xi_1 + \ldots + \xi_n) - \cF G_{s-t_n} (\xi_1 + \ldots + \xi_n) \Big) dx d\pmb{t}\nonumber \\
&= \cF \mathbf{1}_{B_R} (\xi_1 + \ldots + \xi_n) \int_{T_n(t)} \prod_{j=1}^{n-1} \cF G_{t_{j+1}-t_j} (\xi_1 + \ldots + \xi_j) \nonumber \\
	& \qquad \qquad \qquad \qquad \Big( \cF G_{t-t_n} (\xi_1 + \ldots + \xi_n) - \cF G_{s-t_n} (\xi_1 + \ldots + \xi_n) \Big)d\pmb{t}.
\end{align}

By the triangle inequality and the hypercontractivity property, we have
\begin{align} \label{ineq-||F_R(t)-F_R(s)||}
	\|F_R(t)-F_R(s)\|_p
	\leq \sum_{n\geq 1} (p-1)^{n/2} \| I_n (g_{n,R}(\cdot;t,s)) \|_2
	= \sum_{n\geq 1} (p-1)^{n/2} \left( n! \left\| \widetilde{g}_{n,R}(\cdot;t,s) \right\|_{\cP_0^{\otimes n}}^2 \right)^{1/2}.
\end{align}
Using the Fourier transform for expressing the inner product, \eqref{eq-Fourier of g} and Lemma \ref{Lem-inner product inequality} with $\mu_n(d\xi_1 \ldots d\xi_n) = |\cF \mathbf{1}_{B_R} (\xi_1 + \ldots + \xi_n)|^2 \mu(d\xi) \ldots \mu(d\xi_n)$, we have
\begin{align}
\label{eq-norm of tilde g}
& n! \| \widetilde{g}_{n,R}(\cdot;t,s)\|_{\cP_0^{\otimes n}}=n!\langle g_{n,R}(\cdot;t,s), \widetilde{g}_{n,R}(\cdot;t,s) \rangle_{\cP_0^{\otimes n}} \nonumber \\
\nonumber
& \quad =n! \int_{(\bR^d)^n} \cF g_{n,R}(\cdot;t,s)(\xi_1,\ldots,\xi_n)
\overline{\cF \widetilde{g}_{n,R}(\cdot;t,s)(\xi_1,\ldots,\xi_n)}\mu(d\xi_1) \ldots \mu(d\xi_n)\\
& \quad =\sum_{\rho \in S_n} \int_{(\bR^d)^n} \cF g_{n,R}(\cdot;t,s)(\xi_1,\ldots,\xi_n)
	\overline{\cF g_{n,R}(\cdot;t,s) (\xi_{\rho(1)},\ldots,\xi_{\rho(n)})} \mu(d\xi_1) \ldots \mu(d\xi_n) \nonumber \\
& \quad =\sum_{\rho \in S_n} \int_{(\bR^d)^n}
	\int_{T_n(t)} \int_{T_n(t)}
	\prod_{j=1}^{n-1} \cF G_{t_{j+1}-t_j} (\xi_1 + \ldots + \xi_j) \prod_{j=1}^{n-1} \overline{\cF G_{t_{j+1}'-t_j'} (\xi_{\rho(1)} + \ldots + \xi_{\rho(j)})}   \nonumber \\
	& \quad \quad \quad \Big( \cF G_{t-t_n} (\xi_1 + \ldots + \xi_n) - \cF G_{s-t_n} (\xi_1 + \ldots + \xi_n) \Big) \nonumber \\
	& \quad \quad \quad \overline{\Big( \cF G_{t-t_n'} (\xi_{\rho(1)} + \ldots + \xi_{\rho(n)})-\cF G_{s-t_n'} (\xi_{\rho(1)} + \ldots + \xi_{\rho(n)}) \Big)} d\pmb{t} d\pmb{t'} \nonumber \\
& \quad \quad \quad  |\cF \mathbf{1}_{B_R} (\xi_1 + \ldots + \xi_n)|^2 \mu(d\xi_1) \ldots \mu(d\xi_n)\nonumber \\
	& \quad \leq t^n \int_{(\bR^d)^n} \mu(d\xi_1) \ldots \mu(d\xi_n) |\cF \mathbf{1}_{B_R} (\xi_1 + \ldots + \xi_n)|^2 \int_{T_n(t)} d\pmb{t}
	 \nonumber \\
	& \quad \left| \prod_{j=1}^{n-1} \cF G_{t_{j+1}-t_j} (\xi_1 + \ldots + \xi_j) \right|^2 \left| \cF G_{t-t_n} (\xi_1 + \ldots + \xi_n) - \cF G_{s-t_n} (\xi_1 + \ldots + \xi_n) \right|^2.
\end{align}

Noting that $\cF G_t(\xi) = |\xi|^{-1} \sin (t|\xi|)$ is bounded by $t$ and is a $1$-Lipschitz function in the time variable $t$, uniformly over $\xi \in \bR^d$, we have
\begin{align} \label{ineq-Fourier-1}
	\big| \cF G_{t-t_n} (\xi_1 + \ldots + \xi_n) - \cF G_{s-t_n} (\xi_1 + \ldots + \xi_n) \big|
	\le \begin{cases}
		t-s, \ \mathrm{if} \ t_n \le s, \\
		t-t_n, \ \mathrm{if} \ s< t_n <t
	\end{cases}
	\le t-s.
\end{align}
Substituting \eqref{ineq-Fourier-1} to \eqref{eq-norm of tilde g} and using \cite[Lemma 2.6]{BNQZ}, the fact that $|\cF G_t(\xi)|^2 =\frac{\sin^2(t|\xi|)}{|\xi|^2}\leq D_t \frac{1}{1+|\xi|^2}$ with $D_t=2(t^2 \vee 1)$, and \eqref{max-principle},
we obtain
\begin{align} \label{ineq-g norm bound}
	&n! \left\| \widetilde{g}_{n,R}(\cdot;t,s) \right\|_{\cP_0^{\otimes n}}^2 \nonumber
	\le (t-s)^2 t^n \int_{(\bR^d)^n}\int_{T_n(t)}  |\cF \mathbf{1}_{B_R} (\xi_1 + \ldots + \xi_n)|^2 \nonumber \\
	& \qquad \qquad \qquad \qquad \qquad \times \left| \prod_{j=1}^{n-1} \cF G_{t_{j+1}-t_j} (\xi_1 + \ldots + \xi_j) \right|^2  dt_1\ldots dt_n \mu(d\xi_1) \ldots \mu(d\xi_n) \nonumber \\
	\le& (t-s)^2 t^n \int_{(\bR^d)^n} \mu(d\xi_1) \ldots \mu(d\xi_n) \int_{T_n(t)} dt_1\ldots dt_n |\cF \mathbf{1}_{B_R} (\xi_n)|^2  \left| \prod_{j=1}^{n-1} \cF G_{t_{j+1}-t_j} (\xi_j) \right|^2 \nonumber \\
	\le& (t-s)^2 t^n \int_{T_n(t)} dt_1\ldots dt_n \int_{\bR^d} |\cF \mathbf{1}_{B_R}(\xi)|^2 \mu(d\xi) \left( D_t \int_{\bR^d} \dfrac{\mu(d\xi)}{1+|\xi|^2} \right)^{n-1} \nonumber \\
	=& \dfrac{(t-s)^2 t^{2n} D_t^{n-1} C_{\mu}^{n-1}}{n!} \int_{\bR^d} |\cF \mathbf{1}_{B_R}(\xi)|^2 \mu(d\xi).
\end{align}
Since $\gamma$ is non-negative and belongs to $L^1(\bR^d)$, we have
\begin{align} \label{ineq-fourier-1_B_R}
	\int_{\bR^d} |\cF \mathbf{1}_{B_R}(\xi)|^2 \mu(d\xi)
	= \int_{(\bR^d)^2} \mathbf{1}_{B_R}(x) \mathbf{1}_{B_R}(y) \gamma(x-y) dxdy
	\le C_{\gamma,d}R^{d/2}
\end{align}
where $C_{\gamma,d}=\|\gamma\|_{L^1(\bR^d)} \omega_d$. By \eqref{ineq-||F_R(t)-F_R(s)||}, \eqref{ineq-g norm bound} and \eqref{ineq-fourier-1_B_R}, we obtain
\begin{align*}
	\|F_R(t)-F_R(s)\|_p
	\lesssim \sum_{n\geq 1} (p-1)^{n/2} \dfrac{(t-s) t^n (D_t C_{\mu})^{\frac{n-1}{2}}}{\sqrt{n!}} \sqrt{C_{\gamma,d} R^d}
\leq C(t-s) R^{d/2},
\end{align*}
where $C$ is a positive constant that only depends on $T, \mu, d,p$.

\medskip

{\em Step 2.} (finite-dimensional convergence)
Let $Q_R(t)=R^{-d/2}F_R(t)$. Fix $T>0$. We have to show that for any $m \in \bN_+$, $0\leq t_1<\ldots<t_m\leq T$,
\begin{equation}
\label{fdd-Q}
(Q_R(t_1),\ldots,Q_R(t_m)) \stackrel{d}{\to} (\cG(t_1),\ldots,\cG(t_m))
\end{equation}

By relation \eqref{F-deltaD}, $F_R(t_i)=\delta(-DL^{-1}F_R(t_i))$ and hence,
\begin{equation}
\label{repr-Q}
Q_R(t_i)=\delta(-{R^{-d/2}}DL^{-1}F_R(t_i)).
\end{equation}
Let
\begin{align*}
	\mathcal{C}_{ij}=\bE[\cG(t_i)\cG(t_j)]=\omega_d \int_{\bR^d}\rho_{t_i,t_j}(x)dx
	= \omega_d \sum_{n\geq 1} \dfrac{1}{n!} \int_{\bR^d} \alpha_n(x;t_i,t_j) dx
\end{align*}
We use representation \eqref{repr-Q}. By Theorem 6.1.2 of \cite{NP}, for any continuous function $h:\bR^m \to \bR$ with bounded second derivatives, we have:
\begin{align}
\nonumber
&
\Big|\bE[h(Q_R(t_1),\ldots,Q_R(t_m))]-\bE[h(\cG(t_1),\ldots,\cG(t_m))]\Big| \\
\nonumber
& \quad \quad \quad \leq \frac{m}{2}\|h''\|_{\infty} \sqrt{\sum_{i,j=1}^m\bE\Big|
\langle D Q_R(t_i), -R^{-d/2} DL^{-1} F_R(t_j)\rangle_{\cP_0}-\mathcal{C}_{ij}\Big|^2}\\
\label{th6-1-2}
& \quad \quad \quad = \frac{m}{2}\|h''\|_{\infty} \sqrt{\sum_{i,j=1}^m\bE\Big|
\frac{1}{R^{d}} \langle D F_R(t_i),  -DL^{-1} F_R(t_j)\rangle_{\cP_0}-\mathcal{C}_{ij}\Big|^2}
\end{align}
where for the last line we used the fact that
$DQ_R(t_i)=R^{-d/2} F_R(t_i)$. It suffices to show that for any $i,j=1,\ldots,m$
\begin{equation}
\label{fdd-Q1}
\bE\left|\frac{1}{R^{d}}
\langle D F_R(t_i), -DL^{-1} F_R(t_j) \rangle_{\cP_0}-\mathcal{C}_{ij}\right|^2 \to 0 \quad \mbox{as} \quad R \to \infty.
\end{equation}
Then \eqref{fdd-Q} follows by applying \eqref{th6-1-2} to the function $h(x_1,\ldots,x_m)=\exp(-i\sum_{j=1}^m u_j x_j)$ for arbitrary $u_1,\ldots,u_m \in \bR$.

We now prove \eqref{fdd-Q1}. Fix $i,j\leq m$ and let $X_{R,ij}= R^{-d}
\langle D F_R(t_i), -DL^{-1}F_R(t_j) \rangle_{\cP_0}$. Note that
\[
\bE|X_{R,ij}-\mathcal{C}_{ij}|^2 \leq 2 \{ {\rm Var}(X_{R,ij})+|\bE(X_{R,ij})-\mathcal{C}_{ij}|^2\}
\]
and $\bE(X_{R,ij})=R^{-d} \bE[F_R(t_i)F_R(t_j)] \to \mathcal{C}_{ij}$ as $R \to \infty$, by duality and part (i). We will prove below that
\begin{equation}
\label{Var-Fij}
{\rm Var}\Big(\langle D F_R(t_i), -DL^{-1} F_R(t_j) \rangle_{\cP_0}\Big) \leq C R^d.
\end{equation}
Then ${\rm Var}(X_{R,ij}) \leq C R^{-d}$ and relation \eqref{fdd-Q1} follows.

To prove \eqref{Var-Fij}, we apply a version of Proposition 1.9 of \cite{BNQZ} for the time-independent noise. We obtain:
\[
{\rm Var}\Big(\langle D F_R(t_i), -DL^{-1} F_R(t_j) \rangle_{\cP_0}\Big)  \leq (T_1+T_2),
\]
where
\begin{align*}
T_1&=\int_{(\bR^d)^6} \|D_{z,w}^2 F_R(t_i)\|_4 \|D_{y,w'}^2 F_R(t_i)\|_4
\|D_{z'} F_R(t_j)\|_4 \|D_{y'} F_R(t_j)\|_4 \\
& \qquad \qquad \qquad \gamma(y-y') \gamma(z-z') \gamma(w-w') dydy'dzdz' dwdw'
\end{align*}
and $T_2$ has a similar expression by switching the roles of $F_R(t_i)$ and $F_R(t_j)$. Similarly to \eqref{bound-A-integr}, it can be proved that $T_1 \leq CR^d$ and $T_2 \leq CR^d$. This proves \eqref{Var-Fij}.

\section{Proof of Theorem \ref{CLT-Riesz}}
\label{section-CLT2}

In this section, we give the proof of Theorem \ref{CLT-Riesz}. Recall
that the Riesz kernel $\gamma(x) = |x|^{-\beta}$ with $\beta \in (0,d)$ is the Fourier transform of the measure $\mu(d\xi) = c_{d,\beta} |\xi|^{-(d-\beta)} d\xi$. This is Example \ref{ex}.3.
Condition (D) holds since $\beta<2$.

\subsection{Proof of Theorem \ref{CLT-Riesz}.(i)}
\label{section-proof-i'}

{\em Step 1.}
Recalling \eqref{eq-chaos expansion}, we have the Wiener chaos expansion $F_R(t) = \sum_{n\geq 1} {\bf J}_{n,R} (t)$, with
\begin{align*}
	{\bf J}_{n,R} (t) = I_n(g_{n,R}(\cdot;t))  \quad \mbox{and} \quad g_{n,R}(\cdot;t)=\int_{B_R} f_n(\cdot,x;t) dx.
\end{align*}
By orthogonality of the Wiener chaos spaces,
\[
\bE[F_R(t) F_R(s)]=
\sum_{n\geq 1}\bE \left[ {\bf J}_{n,R} (t) {\bf J}_{n,R} (s)   \right].
\]
We will prove the that only the projection on the first chaos contributes to the limit.

\medskip

{\em Step 2.} We first consider ${\bf J}_{1,R}(t)$. For any $t>0,s>0$, we have
\begin{align*}
	\bE \left[ {\bf J}_{1,R} (t) {\bf J}_{1,R} (s) \right]
	=& \langle g_{1,R}(\cdot;t),g_{1,R}(\cdot;s) \rangle_{\cP_0} \\
	=&\int_{B_R^2} dxdx' \int_{(\bR^d)^2} f_1(x_1,x;t) f_1(x_1',x';s) \gamma(x_1-x_1') dx_1 dx_1' \\
	=& \int_{B_R^2} dxdx' \int_{\bR^d} \cF f_1(\cdot,x;t)(\xi) \overline{\cF f_1(\cdot,x';s)(\xi)} \mu(d\xi) \\
	=& \int_{B_R^2} dxdx' \int_0^t dt_1 \int_0^s dt_1' \int_{\bR^d} e^{-i\xi \cdot (x-x')} \cF G_{t-t_1}(\xi) \overline{\cF G_{s-t_1'}(\xi)} \mu(d\xi).
\end{align*}
Applying the change of variables $(x,x',\xi) \to (Rx,Rx',\xi/R)$, we get
\begin{align*}
	\bE \left[ {\bf J}_{1,R} (t) {\bf J}_{1,R} (s) \right]
	=& R^{2d-\beta} \int_{B_1^2} dxdx' \int_0^t dt_1 \int_0^s dt_1' \int_{\bR^d} e^{-i\xi \cdot (x-x')} \cF G_{t-t_1}(\xi/R) \overline{\cF G_{s-t_1'}(\xi/R)} \mu(d\xi).
\end{align*}
Note that for $r>0$, $\cF G_r(\xi/R)$ is uniformly bounded and convergent to $r$ as $R \to + \infty$. Besides, we have
\begin{align}
\label{eq-Fourier-1}
	\int_{B_1^2} e^{-i\xi \cdot (x-x')} dxdx'
	= \left| \cF \mathbf{1}_{B_1} (\xi) \right|^2.
\end{align}
Hence, by Fubini's theorem and dominated convergence theorem, we have
\begin{align} \label{lim-J_1}
	\lim_{R \to +\infty} \dfrac{\bE \left[ {\bf J}_{1,R} (t) {\bf J}_{1,R} (s) \right]}{R^{2d-\beta}}
	=& \int_0^t (t-t_1) dt_1 \int_0^s (s-t_1') dt_1' \int_{\bR^d} \left| \cF \mathbf{1}_{B_1} (\xi) \right|^2 \mu(d\xi) \nonumber \\
	=& \int_0^t (t-t_1) dt_1 \int_0^s (s-t_1') dt_1' \int_{B_1^2} \gamma(x-x') dxdx' \nonumber \\
	=& \dfrac{t^2s^2}{4} \int_{B_1^2} |x-x'|^{-\beta} dxdx'=:\dfrac{t^2s^2}{4}\kappa_{\beta,d}.
\end{align}

{\em Step 3.} We consider ${\bf J}_{n,R}(t)$ for $n \ge 2$. 
\begin{align*}
	\bE \left[ {\bf J}_{n,R}^2 (t)  \right]&= n! \| \widetilde{g}_n(\cdot;t)\|_{\cP_0^{\otimes n}}^2 =n! \langle  g_n(\cdot;t), \widetilde{g}_n(\cdot;t) \rangle_{\cP_0^{\otimes n}}\\
=& n! \int_{(\bR^d)^n} \cF g_n(\cdot;t)(\xi_1,\ldots,\xi_n)
\overline{\cF g_n(\cdot;t)(\xi_1,\ldots,\xi_n)} \mu(d\xi_1) \ldots \mu(d\xi_n) \\
	=& n!  \int_{B_R^2} \int_{(\bR^d)^n}  \cF f_n(\cdot,x;t) (\xi_1, \ldots, \xi_n) \overline{\cF \widetilde{f}_n(\cdot,x';t)(\xi_1, \ldots, \xi_n)}  \mu(d\xi_1) \ldots \mu(d\xi_n) dxdx'  \\
	=& \sum_{\rho \in \Sigma_n} \int_{B_R^2} dxdx' \int_{(\bR^d)^n} \mu(d\xi_1) \ldots \mu(d\xi_n) \int_{T_n(t)} dt_1 \ldots dt_n \int_{T_n(t)} dt_1' \ldots dt_n' \\
	& \times e^{-i(\xi_1+\ldots+\xi_n)\cdot (x-x')} \prod_{j=1}^n \cF G_{t_{j+1}-t_j}(\xi_1+\ldots+\xi_j) \overline{\cF G_{t_{j+1}'-t_j'}(\xi_{\rho(1)}+\ldots+\xi_{\rho(j)})}.
\end{align*}
Here, we use the convention $t_{n+1} = t_{n+1}' = t$.
Using \eqref{eq-Fourier-1} and Lemma \ref{Lem-inner product inequality} for $\mu_n(d\xi_1,\ldots,d\xi_n) = \left| \cF \mathbf{1}_{B_R} (\xi_1+\ldots+\xi_n) \right|^2 \mu(d\xi_1) \ldots \mu(d\xi_n)$, we obtain
\begin{align*}
	\bE \left[{\bf J}_{n,R}^2 (t) \right]
	\le& t^n \int_{T_n(t)} dt_1 \ldots dt_n \int_{(\bR^d)^n} \left( \int_{B_R^2} e^{-i(\xi_1+\ldots+\xi_n) \cdot (x-x')} dxdx' \right) \mu(d\xi_1) \ldots \mu(d\xi_n) \\
	& \times \prod_{j=1}^n \left| \cF G_{t_{j+1}-t_j}(\xi_1+\ldots+\xi_j) \right|^2.
\end{align*}
Noting that $\mu(d\xi) = c_{d,\beta} |\xi|^{\beta-d} d\xi$, we apply the change of the variables for $\eta_j = \xi_1+\ldots+\xi_j$, and then for $(x,x',\eta_n) \to (Rx,Rx',\eta_n/R)$, we have
\begin{align}
\label{ineq-EJ^2 bound}
	\bE \left[ {\bf J}_{n,R}^2 (t)  \right]
	\le& c_{d,\beta}^n t^n \int_{T_n(t)} dt_1 \ldots dt_n \int_{(\bR^d)^n} \left( \int_{B_R^2} e^{-i\eta_n \cdot (x-x')} dxdx' \right) \nonumber \\
	& \times \prod_{j=1}^n \left| \cF G_{t_{j+1}-t_j}(\eta_j) \right|^2 \prod_{j=1}^n |\eta_j-\eta_{j-1}|^{\beta-d} d\eta_j \nonumber \\
	\le& c_{d,\beta}^n t^n R^{2d-\beta} \int_{T_n(t)} dt_1 \ldots dt_n \int_{(\bR^d)^n} \left( \int_{B_1^2} e^{-i\eta_n \cdot (x-x')} dxdx' \right) \prod_{j=1}^{n-1} \left| \cF G_{t_{j+1}-t_j}(\eta_j) \right|^2 \nonumber \\
	& \times \left| \cF G_{t-t_n}(\eta_n/R) \right|^2 \prod_{j=1}^{n-1} |\eta_j-\eta_{j-1}|^{\beta-d} |\eta_n -\eta_{n-1}R|^{\beta-d} d\eta_1 \ldots d\eta_n.
\end{align}
Here, we use the convention $\eta_0=0$.

We use the fact that $| \cF G_{t-t_n}(\eta_n/R) |^2 \leq t^2$. Then
we integrate $d\eta_n$ using the following fact:
\[
c_{d,\beta}\int_{\bR^d}\int_{B_1^2}e^{-i\eta_n \cdot (x-x')} |\eta_n-R \eta_{n-1}|^{-(d-\beta)}dx dx' d\eta_n=\int_{B_1^2}
e^{-i R\eta_{n-1} \cdot (x-x')} |x-x'|^{-\beta}dxdx'.
\]
See the last part of the proof of part (2) of Proposition 4.1 of \cite{BNQZ}. Hence
\begin{align*}
\frac{1}{R^{2d-\beta}}\bE \left[ {\bf J}_{n,R}^2 (t)  \right] &\leq  t^{n+2}c_{d,\beta}^{n-1} \int_{T_n(t)} \int_{(\bR^d)^{n-1}} \prod_{j=1}^{n-1}|\eta_j-\eta_{j-1}|^{-(d-\beta)} \prod_{j=1}^{n-1}|\cF G_{t_{j+1}-t_j}(\eta_j)|^2\\
& \quad \left( \int_{B_1^2}
e^{-i R\eta_{n-1} \cdot (x-x')} |x-x'|^{-\beta}dxdx'\right) d\eta_1 \ldots d\eta_{n-1}dt_1 \ldots dt_n.
\end{align*}
If $\eta_{n-1}\not=0$, the integral
$\int_{B_1^2} e^{-iR \eta_{n-1} \cdot (x-x')} |x-x'|^{-\beta} dxdx'$ converges to $0$ as $R \to \infty$
by Riemann-Lebesgue's lemma, and is bounded by $\kappa_{\beta,d}$.
Note that the integral
\[
\int_{T_n(t)} \int_{(\bR^d)^{n-1}} \prod_{j=1}^{n-1}|\eta_j-\eta_{j-1}|^{-(d-\beta)} \prod_{j=1}^{n-1}|\cF G_{t_{j+1}-t_j}(\eta_j)|^2d\eta_1 \ldots d\eta_{n-1}dt_1 \ldots dt_n
\]
coincides with the integral $Q_{n-1}$ given by (4.16) of \cite{BNQZ}, and $Q_{n-1}\leq C^{n}/n!$ (see the equation on display after (4.17) of \cite{BNQZ}). Hence
$
\frac{1}{R^{2d-\beta}}\bE \left[ {\bf J}_{n,R}^2 (t)  \right]$ converges to $0$ as $R \to \infty$ and is bounded by
$t^{n+2}c_{d,\beta}^{n-1} \kappa_{\beta,d}\frac{C^n}{n!}$. By the dominated convergence theorem,
\begin{equation}
\label{conv-Jn}
\frac{1}{R^{2d-\beta}}\sum_{n\geq 2}\bE[J_{n,t}^2(R)] \to 0 \quad \mbox{as} \quad R \to \infty.
\end{equation}
By Cauchy-Schwarz inequality and the dominated convergence theorem,
\[
\frac{1}{R^{2d-\beta}}\sum_{n\geq 2} |\bE[J_{n,R}(t)J_{n,R}(s)]|\to 0 \quad \mbox{as} \quad R \to \infty.
\]

\subsection{Proof of Theorem \ref{CLT-Riesz}.(ii)}
\label{section-proof-ii'}

We use again \eqref{dTV-bound}. By part {\em (i)}, $\sigma_R^2(t) \sim K'(t,t) R^{2d-\beta}$. So it is enough to prove that
\begin{equation}
\label{bound-A-Riesz}
\cA \leq C R^{4d-3\beta},
\end{equation}
where $C>0$ is a constant depending on $(t,\gamma,d)$. For this, we use again inequality \eqref{4terms-A}. We examine $\cA_1$.

Noting that $G_t(Rx) = R^{1-d} G_{t/R}(x)$ and $\gamma(Rx) = R^{-\beta} \gamma(x)$, we change the variables
\begin{align}
\label{ch-var}
	(y,y',z,z',w,w',x_1,x_2,x_1',x_2') \to (Ry,Ry',Rz,Rz',Rw,Rw',Rx_1,Rx_2,Rx_1',Rx_2').
\end{align}
We obtain that:
\begin{align*}
\cA_{1}&=R^{6+4d-3\beta}\int_{[0,t]^2} \int_{0<\theta<r<t} \int_{0<\theta'<r'<t}
\int_{B_1^4} \int_{(\bR^d)^6} G_{\frac{t-r}{R}}(x_1-z) G_{\frac{r-\theta}{R}} (z-w) G_{\frac{t-r'}{R}}(x_1'-y) \\
& \qquad \qquad \qquad G_{\frac{r'-\theta'}{R}}(y-w')  G_{\frac{t-s}{R}}(x_2-z') G_{\frac{t-s'}{R}}(x_2'-y') \gamma(y-y')\gamma(z-z')\gamma(w-w') '\\
&  \qquad \qquad \qquad  dydy' dzdz' dwdw' dx_1 dx_2 dx_1' dx_2 dr' d\theta' dr d\theta  dsds'.
\end{align*}

Recall the definition \eqref{Green} of $G$, the integrand above is non-zero only when
\begin{align*}
	|x_1-z|, |x_2-z'|, |x_1'-y|, |x_2'-y'|, |z-w|, |y-w'| \le \dfrac{t}{R}.
\end{align*}
Moreover, the four variables $x_1, x_2, x_1', x_2'$ should be in $B_1$. By triangle inequality, when $R \ge t$, one can deduce that the integral domain $\bR^d$ for the variables $y,y',z,z',w,w'$ can be replaced by the ball $B_{3}$ in $\bR^d$, for all $R \geq t$. Using the following facts
\begin{align*}
	\int_{\bR^d}G_t(x)dx=t \quad \mbox{and} \quad
	\sup_{x \in B_{3}} \int_{B_{3}} \gamma(x-y) dy
	\le \int_{B_{6}} \gamma(y) dy =:D_{\gamma}<\infty,
\end{align*}
when $R \ge t$, we integrate the space variables in the order $x_2', y', x_1', y,w',x_2,z',w,z$, using one function ($G$ or $\gamma$) at a time. The remaining integral $dx_1$ on $B_1$ yields the constant ${\rm Leb}(B_1)=\omega_d$.
The remaining iterated integral in the 6 time variables is bounded by $t^6$. We obtain that:
\[
\cA_1 \leq R^{6+4d-3\beta}C_{\gamma}^3 \left( \frac{t}{R}\right)^6 \omega_d t^6=C R^{4d-3\beta}.
\]
Similarly, it can be proved that $\cA_j \leq C R^{4d-3\beta}$ for $j=2,3,4$. This proves \eqref{bound-A-Riesz}.

\subsection{Proof of Theorem \ref{CLT-Riesz}.(iii)}
\label{section-proof-iii'}

{\em Step 1.} (tightness).
We will prove that
\[
\|F_R(t)-F_R(s)\|_p
	\leq
C(t-s) R^{d-\beta/2},
\]
where $C$ is a positive constant that only depends on $T, \mu, d,p$.
By Kolmogorov's continuity theorem, it will follow that the process
$\{F_R(t)\}_{t\geq 0}$ has a continuous modification.

The formulas \eqref{ineq-||F_R(t)-F_R(s)||} and \eqref{ineq-g norm bound} still hold. Recalling that $\gamma(x) = |x|^{-\beta}$, making change of variables, we have
\begin{align} \label{ineq-fourier-1_B_R-Riesz}
	\int_{\bR^d} |\cF \mathbf{1}_{B_R}(\xi)|^2 \mu(d\xi)
	=& \int_{(\bR^d)^2} \mathbf{1}_{B_R}(x) \mathbf{1}_{B_R}(y) \gamma(x-y) dxdy \nonumber \\
	=& R^{2d} \int_{(\bR^d)^2} \mathbf{1}_{B_1}(x) \mathbf{1}_{B_1}(y) \gamma(Rx-Ry) dxdy \nonumber \\
	=& R^{2d-\beta} \int_{(\bR^d)^2} \mathbf{1}_{B_1}(x) \mathbf{1}_{B_1}(y) \gamma(x-y) dxdy
	=C_{\gamma,d}' R^{2d-\beta}.
\end{align}
Here, $C_{\gamma,d}'$ is a positive constant that only depends on $\gamma,d$. By \eqref{ineq-||F_R(t)-F_R(s)||}, \eqref{ineq-g norm bound} and \eqref{ineq-fourier-1_B_R-Riesz}, we have
\begin{equation}
\label{tight2}
	\|F_R(t)-F_R(s)\|_p
	\leq \sum_{n\geq 1} (p-1)^{n/2} \dfrac{(t-s) t^n (D_t C_{\mu})^{\frac{n-1}{2}}}{\sqrt{n!}} \sqrt{C_{\gamma,d}' R^{2d-\beta}}
\leq
C(t-s) R^{d-\beta/2},
\end{equation}
where $C$ is a positive constant that only depends on $T, \mu, d,p$.

\medskip

{\em Step 2.} (finite-dimensional convergence)
We have to prove that for any $0\leq t_1<\ldots<t_m\leq T$,
\[
\left(\frac{F_R(t_1)}{R^{d-\beta/2}},\ldots,
\frac{F_R(t_m)}{R^{d-\beta/2}}\right)
\stackrel{d}{\to} (\cG(t_1),\ldots,\cG(t_m)) \quad \mbox{as} \quad R \to \infty.
\]
There are two methods for proving this. The first method is similar to the proof of  Theorem \ref{CLT-integr}.(iii) given in Section \ref{section-proof-iii} above, based on the bound:
\[
{\rm Var}\Big(\langle DF_R(t_i), -D L^{-1}F_R(t_j)\rangle_{\cP_0} \Big) \leq C R^{4d-3\beta}.
\]

The second method is faster and uses the domination of the first chaos. We explain this below. Using the chaos expansion of $F_R(t_i)$ for $i=1,\ldots,k$, we write
\begin{align*}
\left(\frac{F_R(t_1)}{R^{d-\beta/2}},\ldots,
\frac{F_R(t_m)}{R^{d-\beta/2}}\right)=\left(\frac{J_{1,R}(t_1)}{R^{d-\beta/2}},\ldots,
\frac{J_{1,R}(t_m)}{R^{d-\beta/2}}\right)+\frac{1}{R^{d-\beta/2}}
\sum_{n\geq 2}(J_{n,R}(t_1),\ldots,J_{n,R}(t_m)).
\end{align*}
The first vector is Gaussian (since it belongs to the first chaos space) and converges in distribution as $R \to \infty$ to $(\cG(t_1),\ldots,\cG(t_m))$ since the covariances converge (as shown in part (i)).
The second term converges to $0$ in $L^2(\Omega;\bR^m)$ as $R \to \infty$, due to \eqref{conv-Jn}.

\vspace{3mm}

\appendix
\section{Auxiliary results}

In this section, we give some auxiliary results which are used in this article.
In the case $d=2$, the function $G$ has the following properties: for any $t>0$ and $x \in \bR^2$,
\begin{equation}
\label{p-norm-G}
\|G_t\|_{L^p(\bR^2)}^p=\frac{(2\pi)^{1-p}}{2-p}t^{2-p}\quad \mbox{for all $p \in (0,2)$},
\end{equation}
\begin{equation}
\label{G-monotone}
G_t^{p}(x) \leq (2\pi t)^{q-p} G_t^{q}(x) \quad \mbox{for all $0<p<q$}
\end{equation}
\begin{equation}
\label{I-less-G}
1_{\{|x|<t\}}\leq 2\pi t G_t(x).
\end{equation}

The following theorem is analogous to Lemma 2.5 of \cite{BNQZ}.

\begin{lemma} \label{Lem-inner product inequality}
Let $\mu_n$ be a symmetric measure on $(\bR^d)^n$ for some integer $n\geq1 $. Then for any $0<s\leq t$,
\begin{align*}
	& \sum_{\rho \in \Sigma_n} \int_{T_n(t)} dt_1 \ldots dt_n \int_{T_n(s)} ds_1 \ldots ds_n \int_{(\bR^d)^n} \mu_n(d\xi_1 \ldots d\xi_n) \nonumber \\
	& h(t_1,\ldots,t_n,\xi_1,\ldots,\xi_n) h(s_1,\ldots,s_n,\xi_{\rho(1)},\ldots,\xi_{\rho(n)}) \nonumber \\
	\le& \dfrac{s^n+t^n}{2} \int_{T_n(t)} dt_1 \ldots dt_n \int_{(\bR^d)^n} |h(t_1,\ldots,t_n,\xi_1,\ldots,\xi_n)|^2 \mu_n(d\xi_1 \ldots d\xi_n),
\end{align*}
for any measurable non-negative function $h$ for which the above integral makes sense.
\end{lemma}

\begin{proof}
Using the inequality $2ab \le |a|^2 + |b|^2$ and the symmetry of $\mu_n$, we have
\begin{align*}
	& \sum_{\rho \in \Sigma_n} \int_{T_n(t)} dt_1 \ldots dt_n \int_{T_n(s)} ds_1 \ldots ds_n \int_{(\bR^d)^n} \mu_n(d\xi_1 \ldots d\xi_n) \nonumber \\
	& h(t_1,\ldots,t_n,\xi_1,\ldots,\xi_n) h(s_1,\ldots,s_n,\xi_{\rho(1)},\ldots,\xi_{\rho(n)}) \nonumber \\
	\le& \dfrac{1}{2} \sum_{\rho \in \Sigma_n} \int_{T_n(t)} dt_1 \ldots dt_n \int_{T_n(s)} ds_1 \ldots ds_n \int_{(\bR^d)^n} |h(s_1,\ldots,s_n,\xi_{\rho(1)},\ldots,\xi_{\rho(n)})|^2 \mu_n(d\xi_1 \ldots d\xi_n) \nonumber \\
	&+ \dfrac{1}{2} \sum_{\rho \in \Sigma_n} \int_{T_n(t)} dt_1 \ldots dt_n \int_{T_n(s)} ds_1 \ldots ds_n \int_{(\bR^d)^n} |h(t_1,\ldots,t_n,\xi_1,\ldots,\xi_n)|^2 \mu_n(d\xi_1 \ldots d\xi_n) \nonumber \\
	\leq & \dfrac{s^n+t^n}{2} \int_{T_n(t)} dt_1 \ldots dt_n \int_{(\bR^d)^n} |h(t_1,\ldots,t_n,\xi_1,\ldots,\xi_n)|^2 \mu_n(d\xi_1 \ldots d\xi_n).
\end{align*}
\end{proof}

\noindent \footnotesize{{\em Acknowledgement.} The authors are grateful Guangqu Zheng for pointing out a simplified argument for proving parts {\em (ii)} of Theorems \ref{CLT-integr} and \ref{CLT-Riesz}, based on Proposition 1.8 of \cite{BNQZ}, instead of the classical Stein-Malliavin bound which was used in the first version of the manuscript.

\normalsize{

\end{document}